\documentclass[reqno,twoside,english,11pt]{article}

\usepackage{amsmath,amsfonts,calrsfs,fullpage,amssymb,xcolor,verbatim,eucal,yfonts,mathrsfs}

\usepackage[english,french]{babel}

\usepackage{mathtools}

\usepackage{color}
\usepackage{babel}
\usepackage{mathrsfs}
\usepackage{amsthm}
\usepackage[unicode=true,bookmarks=true,bookmarksnumbered=true,bookmarksopen=false,breaklinks=false,pdfborder={0 0 1},backref=false,colorlinks=true]{hyperref}
\hypersetup{
 pdfauthor={G Xi}}

\newtheorem{Theorem}{Theorem}[section]
\newtheorem{Definition}[Theorem]{Definition}
\newtheorem{Proposition}[Theorem]{Proposition}
\newtheorem{Lemma}[Theorem]{Lemma}

\newtheorem{Remark}[Theorem]{Remark}

\newtheorem{Hypothesis}{Hypothesis}

\numberwithin{equation}{section}
\theoremstyle{plain}
\newtheorem{thm}{\protect\theoremname}[section]
\theoremstyle{plain}
\newtheorem{prop}[thm]{\protect\propositionname}
\theoremstyle{plain}
\newtheorem{lem}[thm]{\protect\lemmaname}
\theoremstyle{remark}
\newtheorem{rem}[thm]{\protect\remarkname}
\theoremstyle{plain}

\theoremstyle{remark}
\newtheorem*{rem*}{\protect\remarkname}

\def\r2{\mathbb{R}^2}
\def\le{\left}
\def\r{\right}

\def\a{\alpha}

\def\e{\epsilon}

\def\A{{\mathcal A}}
\def\L{{\mathcal L}}

\def\F{{\mathcal F}}

\def\H{{\mathcal H}}

\def\M{{\mathcal M}}

\def\s0t{\sup_{t \in [0,T]}}

\def\ds{\displaystyle}

\def\beq{\begin{equation}}
\def\eeq{\end{equation}}
\def\barr{\begin{array}}
\def\earr{\end{array}}
\def\vs{\vspace{.01mm}   \\}
\def\rd{\reals\,^{d}}

\newcommand{\E}{\mathbb E}

\newcommand{\norm}[1]{\left\| #1\right\|}

\newcommand{\Inner}[1]{\big\langle #1\big\rangle}

\newcommand{\abs}[1]{\lvert #1\rvert}

\def\H{{\mathcal{H}}}

\def\E{{\mathbb{E}}}
\def\P{{\mathbb{P}}}

\usepackage{babel}
\allowdisplaybreaks[4]

\date{}

  \providecommand{\corollaryname}{Corollary}
  
  \providecommand{\lemmaname}{Lemma}
\providecommand{\theoremname}{Theorem}

\theoremstyle{plain}

\makeatother

\providecommand{\corollaryname}{Corollary}
\providecommand{\lemmaname}{Lemma}
\providecommand{\propositionname}{Proposition}
\providecommand{\remarkname}{Remark}
\providecommand{\theoremname}{Theorem}

\begin{document}

\global\long\def\divg{{\rm div}\,}%

\global\long\def\curl{{\rm curl}\,}%

\global\long\def\rt{\mathbb{R}^{3}}%

\global\long\def\rd{\mathbb{R}^{d}}%

\global\long\def\rtwo{\mathbb{R}^{2}}%

\global\long\def\e{\epsilon}%

\title{The small-mass limit for some constrained   wave equations with nonlinear conservative  noise}

\author{Sandra Cerrai\thanks{Department of Mathematics, University of Maryland, cerrai@umd.edu. Partially supported by the NSF grant  DMS-1954299 - {\em Multiscale analysis of infinite-dimensional stochastic systems}.}, \ Mengzi Xie\thanks{Department of Mathematics, University of Maryland, mxie2019@umd.edu.}}

\maketitle

\selectlanguage{english}

\date{}

\maketitle

\begin{abstract} We study the small-mass limit, also known as the Smoluchowski-Kramers diffusion approximation (see \cite{kra} and \cite{smolu}), for a system of  stochastic damped wave equations, whose solution is constrained to live in the unitary sphere of the space of square integrable functions on the  interval $(0,L)$. The stochastic perturbation is given by a nonlinear multiplicative Gaussian noise, where the stochastic differential is understood in Stratonovich sense. Due to its particular structure, such noise  not only conserves $\mathbb{P}$-a.s. the constraint, but also preserves a suitable energy functional. In the limit we derive a deterministic system, that remains confined to the unit sphere of $L^2$, but includes additional terms. These terms depend on the reproducing kernel of the noise and account for the interaction between the constraint and the particular conservative noise we choose. 
\end{abstract}

\section{Introduction}

In recent years, there has been a considerable research activity related to the Smoluchowski-Kramers diffusion approximation for infinite-dimensional systems. The first results in this direction dealt with the case of  constant damping term, with smooth noise and regular coefficients (see \cite{CF1}, \cite{CF2}, \cite{salins}, and \cite{Lv2}). More recently, the case of constant friction has been studied in \cite{Fui} and \cite{zine} for equations perturbed by  space-time white noise in dimension $d=2$, and in \cite{Han} for equations with H\"older continuous coefficients in dimension $d=1$. In all these papers, the fact that the damping coefficient is constant leads to a perturbative result, in the sense that, in the small-mass limit, the solution $u_\mu$ of the stochastic damped wave equation converges to the solution of the stochastic parabolic problem formally obtained by taking $\mu=0$.  
The case of SPDEs with 
state-dependent damping  was considered first in \cite{CX} for a single equation and later, in \cite{CD}, in the case of  systems of equations (see also \cite{CXIE}). Notably, this scenario differs drastically from the previous one, as the non-constant friction leads to an additional  noise-induced term in the small-mass limit. 

An analogous phenomenon has been identified  
 in \cite{BC23}, where  the case of SPDEs constrained to live on a manifold in the functional space of square-integrable functions  $L^2$ was considered. The study of deterministic and stochastic constrained PDEs is not a new field of study. In this context, we would like to mention   the paper \cite{Rybka_2006} by Rybka and the paper \cite{Caff+Lin_2009} by Caffarelli and Lin,  where, in order  to find a gradient flow approach to a specific minimization problem, deterministic heat flows in Hilbert manifolds were explored.
A constrained version the deterministic 2-D Navier-Stokes equation was studied in \cite{Brz+DM_2018} by Brzezniak, Dhariwal and Mariani, as well as in \cite{CPR_2009} by Caglioti, Pulvirenti, and Rousset, and later its stochastic version was investigated in \cite{Brz+DHM_2018} and \cite{Brz+Dhariwal_2021}    by Brzezniak and Dhariwal.  
 
 In \cite{BC23} we have introduced for the first time  a class of damped stochastic wave equations constrained to evolve within the unitary sphere of $L^2$ and we have shown that   the Smoluchowski-Kramers approximation  leads to a stochastic parabolic problem, whose solution is still confined to the unitary $L^2$-sphere and where, as in \cite{CX} and  \cite{CD}, an additional drift term appears. Somewhat surprisingly, such extra drift does not account for the Stratonovich-to-It\^o correction term. 
\medskip

 In the present paper we continue the work started in \cite{BC23} and we introduce the following system of stochastic wave equations on the interval $(0,L)$
\begin{equation}\label{SPDE}
	\le\{\begin{array}{l}
		\ds{\mu\partial_{t}^{2}u_{\mu}(t,x)+\mu\abs{\partial_{t}u_{\mu}(t)}_{L^2(0,L)}^{2}u_{\mu}(t,x), }\\
		\vs 
		\ds{\quad\quad =\partial^2_x u_{\mu}(t,x)+\abs{\partial_x u_{\mu}(t)}_{L^2(0,L)}^{2}u_{\mu}(t,x)-\gamma\partial_{t}u_{\mu}(t,x)+\sqrt{\mu}\,\big(u_{\mu}(t)\times\partial_{t}u_{\mu}(t)\big)\circ\partial_{t}w(t,x)  }\\
		\vs 
		\ds{u_{\mu}(0,x)=u_{0}(x),\ \ \ \partial_{t}u_{\mu}(0,x)=v_{0}(x),\ \ \  u_{\mu}(t,0)=u_{\mu}(t,L)=0, }
	\end{array}\r.
\end{equation}
depending on a parameter $0<\mu<<1$. Here $u_\mu(t,x) \in\,\mathbb{R}^3$, for every $(t,x) \in\,[0,+\infty)\times (0,L)$, the friction coefficient $\gamma$ is strictly positive, and $w(t)$ is a cylindrical Wiener process, white in time and colored in space, defined on some stochastic basis $(\Omega,\mathcal{F},\{\mathcal{F}_t\},\mathbb{P})$, with $\circ$ denoting the Stratonovich stochastic differential. The solution $u_\mu(t)$ is subject to the finite-codimension constraint of living on $M=S_{L^2(0,L)}(0,1)$, the unitary sphere of $L^2(0,L)$, with the initial data $(u_0,v_0)$ in $\mathcal{M}$, the tangent bundle of $M$. 

The key and fundamental distinction between the present paper and \cite{BC23} lies in the nature of the random perturbation  considered. Actually,  unlike any previous work related to the Smoluchowski-Kramers diffusion-approximation, both in finite and infinite dimensions, here we consider a diffusion coefficient $\sigma_\mu$
  which is nonlinear and includes both  the position $u_\mu(t)$ and the velocity $\partial_t u_\mu(t)$, through the vector product  $\sqrt{\mu}\,u_\mu(t)\times \partial_t u_\mu(t)$.
  The reason why in all previous works the diffusion does not depend on the velocity is that, while one expects a limit for $u_\mu$, there is no limit for $\partial_t u_\mu$, and it is not clear how to make sense of the limit in the equation, especially when it comes to the martingale term. However, as shown in previous work, the term $\sqrt{\mu}\,\partial_t u_\mu$ can have non trivial limiting behavior, and with this new work we are trying to understand what happens when $\sigma_\mu(u)=\sqrt{\mu}\, u\times v$.

Since in \cite{BC23} the diffusion coefficient did not include the velocity $\partial_t u_\mu(t)$, the It\^o and Stratonovich interpretations of the stochastic differential yielded the same equation. In the current setting, however, the It\^o and Stratonovich differentials lead to different equations, and our choice to interpret the stochastic differential in the Stratonovich sense has significant implications. Due to  the special structure of the diffusion coefficient, both the Stratonovich and It\^o integrals ensure that the solution  $(u_\mu(t), \partial_tu_\mu(t))$ remains within  the tangent bundle $\mathcal{M}$, for every $t\geq 0$. However, the noise in the Stratonovich sense exhibits a more substantial conservative behavior by preserving also the energy
\begin{equation}\label{intro3}
\mathcal{E}_\mu(t):=\vert u_\mu(t)\vert_{H^1_0(0,L)}^2+\mu\,\vert \partial_t u_\mu(t)\vert_{L^2(0,L)}^2+\int_0^t \vert \partial_t u_\mu(s)\vert_{L^2(0,L)}^2\,ds,	
\end{equation}
almost surely with respect to $\mathbb{P}$. 
This phenomenon, well-understood in other contexts, particularly in the parabolic setting, plays a critical role in the scenario considered here, as it serves as a key tool in proving the necessary bounds for $u_\mu(t)$ and $\sqrt{\mu}\,\partial_t u_\mu(t)$ in the appropriate functional spaces, uniformly with respect to $\mu \in (0,1)$. And those bounds are fundamental in the proof of the tightness and  in the identification of the limit.

After showing that for every fixed $\mu \in\,(0,1)$ and $p\geq 1$, and any initial condition $(u_0,v_0) \in\,[H^1_0(0,L)\times L^2(0,L)]\cap \mathcal{M}$ and $p\geq 1$ there exists a unique mild solution 
\[z_\mu=(u_\mu,\partial_tu_\mu) \in\,L^p(\Omega;C([0,+\infty);[H^1_0(0,L)\times L^2(0,L)]\cap \mathcal{M})),\]
 we study the limiting behavior of $u_\mu$, as $\mu\downarrow 0$. Our main result consists in proving that if $(u_0,v_0) \in\,[H^2(0,L)\times H^1(0,L)]\cap \mathcal{M}$, then, for every $T>0$ and $\delta<2$ and for every $\eta>0$ we have 
	\begin{equation}\label{intro4}
		\lim_{\mu\to0} \P\Big(\abs{u_{\mu}-u}_{C([0,T];H^{\delta})}>\eta\Big)=0.
	\end{equation}
Here $u$ is the unique solution of the  deterministic problem
	\begin{equation}\label{intro1}
	\le\{\begin{array}{l}
		\ds{\partial_{t}\Big[\Big(\gamma+\frac{1}{2}\varphi\abs{u(t)}^{2}\Big)u(t)\Big]=\partial^2_x u(t) +  \abs{\partial_x u(t)}_{H}^{2}u(t) +\frac{3\varphi}{2\gamma}\left(\left[\partial^2_x u(t)+\abs{\partial_x u(t)}_{H}^2u(t)\right] \cdot u(t)\right)u(t),}\\
		\vs
		\ds{u(0,x)=u_{0}(x),\ \ \ \ u(t,0)=u(t,L)=0, }
	\end{array}\r.
\end{equation}
where 
\[\varphi(x)=\sum_{i=1}^\infty \xi_i^2(x),\ \ \ \ \ \ \ x \in\,[0,L],\]
and $\{\xi_i\}_{i \in\,\mathbb{N}}$ is an orthonormal basis for the reproducing kernel $K$ of the noise $w(t)$. 
In particular, this means that $u_\mu$ converges to a deterministic limit $u$, which solves a deterministic problem,  where the constraint to stay on the unitary sphere of $L^2(0,L)$ is preserved. Remarkably, as in the previously mentioned cases - where however only stochastic limits are obtained - in the small-mass limit several noise-induced terms appear in the limiting equation, and such terms depend on the noise  present in the  second-order problem through the function $\varphi$.

It is important to remark that this non trivial behavior emerges only in the $\sqrt{\mu}$ scaling for the diffusion coefficient, as in the case of $\mu^\a$, with $\a>1/2$, the limiting equation  \eqref{intro1} has to be replaced by the constrained parabolic problem
\[\gamma\partial_{t}u(t)=\partial^2_x u(t) +  \abs{\partial_x u(t)}_{H}^{2}u(t),\]
with the same initial and boundary conditions, where there are no noise-induced terms. As for the limiting behavior of $u_\mu$ in the scaling $\mu^\alpha$, with $\alpha \in\,[0,1/2)$, at this stage it is not clear what we should expect. We believe that if any limiting point exists, it should satisfy the deterministic equation
\begin{equation}
\label{intro2}	
\abs{u(t)}^{2}u(t)=\abs{u_0}^{2}u_{0} +\frac{3}{\gamma}\int_{0}^{t}\big(\partial^2_x u(s)\cdot u(s)\big)u(s)ds +\frac{3}{\gamma}\int_{0}^{t}\abs{\partial_x u(s)}_{H}^{2}\abs{u(s)}^{2}u(s)ds.
\end{equation}
However, in order to prove the convergence of any limiting point for the family $\{u_\mu\}_{\mu \in\,(0,1)}$ to a solution of \eqref{intro2}, tightness in at least  $L^2(0,T;H^1_0(0,L))$ would be necessary, and, because of the nature of the diffusion coefficient $\sigma(u,v)=u\times v$, the uniform bounds required for its proof  seem to be out of reach. 

Under the $\sqrt{\mu}$ scaling assumption, in addition to the energy identity \eqref{intro3}, we can prove suitable bounds for $u_\mu$ and $\partial_t u_\mu$ in spaces of higher regularity than $C([0,T];H^1_0(0,L))\cap L^2(0,T;H^2(0,L))$ for $u_\mu$ and $L^2(0,T;L^2(0,L))$ for $\partial_t u_\mu$, which are uniform with respect to $\mu \in\,(0,1)$. Those bounds allow to show that $\{\mathcal{L}(u_\mu)\}_{\mu \in\,(0,1)}$ is tight in $C([0,T];H^\delta(0,L))$, for every $\delta<2$, and  any weak limit point for $u_\mu$ is a solution of \eqref{intro1}. Equation \eqref{intro1} is highly non trivial and the existence of solutions is obtained only as a consequence of the small-mass limit. However, any limiting point for $\{u_\mu\}_{\mu \in\,(0,1)}$ turns out to belong to the space of functions in $C([0,T];H^1_0(0,L))\cap L^2(0,T;H^2(0,L))$, which admit a weak derivative in time in $L^2(0,T;H)$. What is important is that despite its complex form, we can prove the uniqueness of the solution to equation \eqref{intro1} in these functional spaces. Consequently, we can identify any limit point for $\{u_\mu\}_{\mu \in\,(0,1)}$ with the unique solution of  equation \eqref{intro1} and limit \eqref{intro4} follows.

\medskip

Before concluding, let us outline the structure of this paper. In Section \ref{sec2}, we introduce all the assumptions and notations that will be used throughout the paper. Section \ref{sec3} presents the main results. In Section \ref{sec4}, we study the well-posedness of equation \eqref{SPDE}, and in Section \ref{sec5}, we establish bounds for the solution  $u_\mu$ and $\sqrt{\mu}\,\partial_t u_\mu$ which hold uniformly with respect to $\mu \in\,(0,1)$.  Section \ref{sec6} focuses on the limiting equation \eqref{intro1}; we introduce an equivalent formulation and prove the uniqueness of the solution in a suitable functional space. Finally, in Section \ref{sec7}, we demonstrate the validity of limit \eqref{intro4}. This is achieved by first integrating \eqref{intro1} with respect to time and rearranging all terms in a proper way, and then by proving tightness and identifying any weak limit as the unique solution of \eqref{intro1}.

\section{Notations and assumptions}
\label{sec2}

Let $H$ denote the Hilbert space $L^{2}(0,L;\mathbb{R}^3)$, for some fixed  $L>0$, endowed with the  inner product 
\[\Inner{u,v}_{H}=\sum_{i=1}^3\langle u_i,v_i\rangle_{L^2(0,L)}=\int_0^L(u(x)\cdot v(x))\,dx,\] 
and the corresponding norm $\abs{\cdot}_{H}$. Notice that here and in what follows, we shall denote the scalar product of two vectors $h, k \in\,\mathbb{R}^3$ by $(h\cdot k)$. Moreover we shall denote the norm of a vector $h$ in $\mathbb{R}^3$ by $\vert h\vert_{\mathbb{R}^3}$, or just by $\vert h\vert$, when there is no risk of confusion.

 Next, for every $k \in\,\mathbb{N}$ we shall denote by $H^k$ the closure of $C^\infty_0([0,L])$ in $W^{k,2}(0,L)$, where $W^{k,2}(0,L)$ is the space of all functions $u \in\,H$ such that $D^h u$ exists in the weak sense, for every $h\leq k$, and $D^h u \in\,H$. Due to the Poincar\'e inequality, we can endow $H^k$ with the norm
\[\vert u\vert_{H^k}:=\vert D^k u\vert_{H}.\]
Moreover, we shall set $\H_{k}:=H^{k+1}\times H^{k}$. When $k=0$, we will simply denote $\H_0$ by $\H$. Finally, for every function $u:(0,L)\to\mathbb{R}^3$ we shall denote
\[\vert u\vert_\infty=\sup_{x \in\,[0,L]}|u(x)|_{\mathbb{R}^3}.\]
Notice that since $[0,L]\subset \mathbb{R}$, we have $H^1\hookrightarrow L^\infty(0,L)$.

If $E$ and $F$ are Banach spaces,  the class of all bounded linear operators from $E$ to $F$  will be denoted by $\mathcal{L}(E,F)$. We will use a shortcut notation
$\mathcal{L}(E)$ for $\mathcal{L}(E,E)$. It is known that $\mathcal{L}(E,F)$ is also a Banach space. By $\mathcal{L}_2(E,E;F)$ we will
denote the Banach space of all bounded bilinear operators from $E\times E=:E^2$ to $F$. If $K$ is another Hilbert space,
	by $\mathcal{T}_2(K,F)$,  we will denote  the Hilbert space of all Hilbert-Schmidt operators from $K$ to $F$, endowed with the natural inner product and norm.
It is known that $\mathcal{T}_2(K,F) \hookrightarrow \mathcal{L}(K,F)$ continuously.
If $\{e_j\}_{j \in \mathbb{N}}$ is   an orthonormal basis of a separable Hilbert space $K$ which  is continuously embedded into a Banach space $E$ and
\[
\sum_{j=1}^{\infty} \vert e_j\vert_E^2 <\infty,
\]
then for every $\Lambda \in \mathcal{L}_2(E\times E;F)$ we put
\begin{equation}\label{eqn-tr K}
\text{tr}_K(\Lambda)= \sum_{i=1}^{\infty} \Lambda(e_j,e_j).
\end{equation}

In what follows, we denote by $M$ the unit sphere in $H$
\begin{equation*}
	M=\big\{u\in H:\abs{u}_{H}=1\big\},
\end{equation*}
and by $\M$ the corresponding tangent bundle
\begin{equation*}
	\M=\big\{(u,v)\in M\times H:\Inner{u,v}_{H}=0 \big\}.
\end{equation*}

\medskip

Now, we rewrite equation \eqref{SPDE}  as the following system
\begin{equation}\label{SPDE1}
	\le\{\begin{array}{l}
		\ds{du_{\mu}(t)=v_{\mu}(t)dt, }\\
		\vs 
		\ds{dv_{\mu}(t)=\frac{1}{\mu}\left[\partial^2_x u_{\mu}(t)+\abs{\partial_x u_{\mu}(t)}_{H}^{2}u_{\mu}(t)-\mu\abs{v_{\mu}(t)}_{H}^{2}u_{\mu}(t)-\gamma v_{\mu}(t)\right]dt}
		\\
		\vs
		\ds{ \quad \quad \quad \quad \quad \quad \quad +\frac{1}{\sqrt{\mu}}\big(u_{\mu}(t)\times v_{\mu}(t)\big)\circ dw(t),  }\\
		\vs
		\ds{u_{\mu}(0)=u_{0},\ \ \ v_{\mu}(0)=v_{0},\ \ \  \ \ \ \ \ \ \ \ u_{\mu}(t,0)=u_{\mu}(t,L)=0 }.
	\end{array}\r.
\end{equation}
Here $w(t)$ is a cylindrical Wiener process in $H$. Thus, if we denote by 
 $K$  its reproducing kernel Hilbert space,  we have
 \[w(t,x)=\sum_{i=1}^\infty \xi_i(x)\beta_i(t),\ \ \ \ \ \ \ \ \ \ \ \ (t,x) \in\,[0,+\infty)\times (0,L),\]
 where $\{\beta_i(t)\}_{i \in\,\mathbb{N}}$ is a sequence of independent standard Brownian motions, all defined on the stochastic basis $(\Omega, \mathcal{F}, \{\mathcal{F}_t\}_{t\geq 0}, \mathbb{P})$, and  $\{\xi_{i}\}_{i\in\mathbb{N}}$ is an orthonormal basis of $K$. In what follows we shall denote by$E$  a Banach space containing $K$, such that the embedding of $K$ in $E$ is Hilbert-Schmidt. In particular
 \[\sum_{i=1}^\infty \vert \xi_i\vert_E^2<\infty.\] 
 
Moreover,  we shall assume the following conditions are  satisfied.
 \begin{Hypothesis}
 \label{H1}	
 All functions  $\xi_{i}$ belong to $C^{1}([0,L])$. Moreover, if we denote 
\begin{equation*}
	\varphi(x):=\sum_{i=1}^{\infty}\abs{\xi_{i}(x)}^{2},\ \ \ \ \ \varphi_{1}(x):=\sum_{i=1}^{\infty}\abs{\xi_{i}'(x)}^{2},\ \ \ \ \ \ x\in (0,L),
\end{equation*}
we have that $\varphi$ and $\varphi_1$ belong to $L^\infty(0,L)$.
 \end{Hypothesis}

\begin{Remark}
{\em Since for every $x\in(0,L)$, we have
\begin{equation*}
	\big\lvert\sum_{i=1}^{\infty}\xi_{i}(x)\xi_{i}'(x)\big\rvert\leq \sqrt{\varphi(x)\varphi_{1}(x)},
	\end{equation*}
thanks to Hypothesis \ref{H1} we have that $\varphi$ is weakly differentiable and
\[\varphi^\prime(x)=2\sum_{i=1}^\infty \xi_i(x)\xi_i^\prime(x),\ \ \ \ \ \ x \in\,(0,L).\]
In particular, this allows to   conclude that $\varphi \in\,W^{1,\infty}(0,L)$ with
\[\abs{\varphi'}_{\infty}\leq c\,(\abs{\varphi}_\infty+\abs{\varphi_1}_\infty).\]
\hfill$\Box$
}	
\end{Remark}

\section{Main results} \label{sec3}

In what follows, we denote by $A$ the realization in $H$ of the second derivative operator, endowed with Dirichlet boundary conditions, and for all $\mu>0$ define 
\begin{equation*}
	\A_{\mu}z:=\big(v,\mu^{-1}A u\big),\ \ \ \ z=(u,v)\in D(\A_{\mu})=\mathcal{H}_1.
\end{equation*}
Moreover, for every  $\mu>0$  we define\begin{equation*}
	F_\mu(z):= -\mu\,\abs{v}_{H}^{2}u+\abs{\partial_x u}_{H}^{2}u,\ \ \ \ \ \ z=(u,v)\in\H.
\end{equation*}
With these notations, system \eqref{SPDE1} can be rewritten as
\begin{equation}\label{SPDE2}
	dz_\mu(t) = \mathcal{A}_\mu z_\mu(t)dt + \frac 1\mu\,\left(0, F_\mu(z_\mu(t))-\gamma\, v_\mu(t)\right)dt +\frac 1{\sqrt{\mu}}\big(0, u_\mu(t)\times v_\mu(t)\big)\circ dw(t). 
\end{equation}
with $z_{\mu}(0)=z_0=(u_0,v_0)$.

The first result we will prove in this paper is the following well-posedness result for system \eqref{SPDE1} in  $\mathcal{H}$.

\begin{thm}\label{well_posedness_system}
	For every $\mu>0$ and $z_0=(u_0,v_0)\in \H\cap \M$, there exists a unique mild  solution to the stochastic constrained wave equation \eqref{SPDE1}. Namely, there exists a unique $\H\cap \M$-valued continuous and adapted process $z_{\mu}(t)=(u_{\mu}(t),v_{\mu}(t))$, $t\geq 0$,  such that the following hold.
	\begin{enumerate}
		\item The process $u_{\mu}(t)$ has $M$-valued trajectories of class $C^1$ and 
		\begin{equation*}
			v_{\mu}(t)=\partial_{t}u_{\mu}(t),\ \ \ \ t\geq0.
		\end{equation*}
		
		\item The process $z_{\mu}(t)$ satisfies the equation
		\begin{equation*}
			\begin{array}{l}
				\ds{z_{\mu}(t)=\mathcal{S}_{\mu}(t)z_0+\frac{1}{\mu}\int_{0}^{t}\mathcal{S}_{\mu}(t-s)\big(0,F_\mu(z_\mu(s))-\gamma v_{\mu}(s)\big)ds }\\
				\vs 
				\ds{\quad\quad \quad \quad\quad \quad\quad \quad \quad+\frac{1}{\sqrt{\mu}}\int_{0}^{t}\mathcal{S}_{\mu}(t-s)\big(0,(u_{\mu}(s)\times v_{\mu}(s))\circ\,dw(s)\big),  }
			\end{array}
		\end{equation*}
		for every $t\geq 0$, $\P$-almost surely.
		\item The following identity holds for every $t\geq 0$
		\begin{equation}
		\label{sm2}
		\vert u_\mu(t)\vert_{H^1}^2+\mu\,\vert v_\mu(t)\vert_H^2+2\,\gamma \int_0^t \vert v_\mu(s)\vert_H^2\,ds=	\vert u_0\vert_{H^1}^2+\mu\,\vert v_0\vert_H^2,\ \ \ \ \ \mathbb{P}-\text{a.s.}
		\end{equation}

	\end{enumerate}
	
\end{thm}

The second result, which represents the main goal of this paper, concerns  the limiting behavior of the process $u_\mu(t)$, as $\mu\downarrow 0$.

\begin{thm} \label{limite}
	Fix $(u_0,v_0)\in \H_{1}\cap\M$. Then, for every $T>0$ and $\delta<2$ and for every $\eta>0$ we have 
	\begin{equation*}
		\lim_{\mu\to0} \P\Big(\abs{u_{\mu}-u}_{C([0,T];H^{\delta})}>\eta\Big)=0,
	\end{equation*}
	where $u$ is the unique solution of the following deterministic problem
	\begin{equation*}
	\le\{\begin{array}{l}
		\ds{\partial_{t}\Big[\Big(\gamma+\frac{1}{2}\varphi\abs{u(t)}^{2}\Big)u(t)\Big]=\partial^2_x u(t) +  \abs{\partial_x u(t)}_{H}^{2}u(t) +\frac{3\varphi}{2\gamma}\left(\left[\partial^2_x u(t)+\abs{\partial_x u(t)}_{H}^2u(t)\right] \cdot u(t)\right)u(t),}\\
		\vs
		\ds{u(0,x)=u_{0}(x),\ \ \ \ u(t,0)=u(t,L)=0. }
	\end{array}\r.
\end{equation*}
\end{thm}

\section{The well-posedness of system \eqref{SPDE1}}
\label{sec4}

As known (see e.g. \cite[Definition 3.1]{Brz+Elw_2000}), for  an arbitrary function $G: F\to \mathcal{L}(E,F)$ of class $C^1$
\begin{equation}\label{sm1300}
\begin{split}
\int_0^t G(z(s)) \,\circ\,dW(s)&:=\int_0^t G(z(s)) \,dW(s)
+\frac12  \int_0^t \text{tr}_{K} \bigl[G^\prime(z(s))G(z(s))\bigr]  \,ds,
\end{split}
\end{equation}
with $\text{tr}_K$  defined  as in \eqref{eqn-tr K}.
Note that 
\[G^\prime(z)\,G(z)\in \mathcal{L}(E; \mathcal{L}(E,F))\equiv \mathcal{L}_2(E\times E,F),\ \ \ \ \ \ z \in\,F,\] and
\[
G^\prime(z)G(z)(e_1,e_2)=[G^\prime(z)\bigl( G(z)e_1)]e_2, \ \ \ \ \ \ (e_1,e_2) \in E\times E,\ \ \ \ \ z \in\,F.
\]
 This means that $\text{tr}_{K} \bigl[ G^\prime(z)G(z)\bigr]$ is a well defined element of $F$ and satisfies
\[
\text{tr}_K [ G^\prime(z)G(z) ]= \sum_{i=1}^{\infty}[ G^\prime(z)\bigl( G(z)e_j\bigr)]e_j,
\]
where $\{e_j\}_{j \in\,\mathbb{N}}$ is   an orthonormal basis  of  $K$. In particular,   if we take
\[G(u,v)k:=(0,\sigma(u,v)k),\ \ \ \ \ (u,v) \in\,\mathcal{H},\ \ \ \ k \in\,E,\]
for some $\sigma:\mathcal{H}\to \mathcal{L}(E,H^1),$ we have
\[[G^\prime(u,v)G(u,v)](k,h)=\left(0,\partial_v\sigma(u,v)[\sigma(u,v)(k)](h)\right),\]
so that
\begin{equation}\label{eqn-tr K-G'G}\text{tr}_K [ G^\prime(z)G(z) ]= \left(0,\text{tr}_K[\partial_v\sigma(u,v)\sigma(u,v)]\right).
\end{equation}

In what follows, we will take
\[\sigma(u,v)k=(u\times v)k.\]
The following result holds.
\begin{lem}\label{lem-sm1}
	The map $\sigma:\H_{k}\to \mathcal{T}_{2}(K,H^{k})$ is Lipschitz-continuous on balls and has polynomial growth, for $k=0,1$. Moreover, its  Fr\'echet derivative along any direction $v \in\,H^k$ is given by
	\begin{equation*}
		\partial_{v}\sigma(z)y=u\times y,\ \ \ \ z=(u,v)\in \H_{k},\ \ \ \ y\in H^{k},
	\end{equation*}
	and the function 
	\begin{equation}\label{sm1310}
		\H_{k}\ni z=(u,v)\mapsto \text{{\em tr}}_{K}\left[\partial_{v}\sigma(z)\sigma(z)\right]=\text{{\em tr}}_{K}\left[u\times (u\times v)\right]=\varphi \left[u\times (u\times v)\right]\in H^{k},
	\end{equation}
	is Lipschitz-continuous on balls with  polynomial growth.
\end{lem}

\begin{proof}
	{\em Case  $k=0$.} For every $z_{1}=(u_1,v_1)$ and  $z_{2}=(u_2,v_2)$ in $\H$ we have 
	\begin{equation*}
		\begin{array}{l}
			\ds{ \Vert\sigma(z_{1})-\sigma(z_{2})\Vert_{\mathcal{T}_{2}(K,H)}^{2}=\sum_{i=1}^{\infty}\abs{(\sigma(z_1)-\sigma(z_2))\xi_{i}}_{H}^{2}}\\[14pt]
			\ds{\quad \quad =\int_0^L \vert u_1(x)\times v_1(x)-u_2(x)\times v_2(x)|^2\sum_{i=1}^{\infty} |\xi_i(x)|^2\,dx			\leq c\,\vert\varphi\vert_{\infty}\abs{u_{1}\times v_{1}-u_{2}\times v_{2}}_{H}^{2}}\\[16pt]
			\ds{\quad \quad    \leq c\,\vert\varphi\vert_{\infty}\Big( \abs{u_{1}-u_{2}}_{H^1}^{2}\abs{v_{1}}_{H}^{2}+\abs{u_{2}}_{H^1}^{2}\abs{v_{1}-v_{2}}_{H}^{2}\Big)\leq c\,\vert\varphi\vert_{\infty}\big(\abs{z_{1}}_{\H}^{2}+\abs{z_{2}}_{\H}^{2}\big)\abs{z_{1}-z_{2}}_{\H}^{2} }.
		\end{array}
	\end{equation*}
Since $\sigma(0)=0$, this implies that $\sigma:\mathcal{H}\to\mathcal{T}_2(K,H)$ is well defined and locally Lipschitz continuous, with quadratic growth.

For every $u \in\,H^1$ fixed, the mapping $
\sigma(u,\cdot):H\to\mathcal{L}(E,H)$ is linear and its derivative $\partial_v \sigma(u,\cdot)$ is given by
\[\partial_v \sigma(z)\,y=u\times y \in\,\mathcal{L}(E,H),\ \ \ \ \ \ \ z=(u,v) \in\,\mathcal{H},\ \ \ \ y \in\,H.\]
In particular
$\partial_v \sigma(z)\,\sigma(z) \in\,\mathcal{L}(E\times E,H)$
and\[		\text{tr}_{K}\left[\partial_{v}\sigma(z)\sigma(z)\right]=\text{tr}_{K}\left[u\times (u\times v)\right]=\sum_{i=1}^{\infty}\left(u\times(u\times v)\xi_i\right)\xi_i=\varphi\left(u\times(u\times v)\right).
	\]
Now, for every $h, k \in\,\mathbb{R}^3$ we have
\begin{equation}\label{sm20}h\times (h\times k)=-\vert h\vert^2 k+(h\cdot k)h,\end{equation}
so that we can write	
\begin{equation*}\text{tr}_{K}\left[\partial_{v}\sigma(z)\sigma(z)\right]=\varphi(-\abs{u}^{2}v+(u\cdot v)u),\end{equation*}
and
\begin{equation*}
		\begin{array}{ll}
			&\ds{ \abs{\text{tr}_{K}\left[u\times (u\times v)\right]}_{H}^{2} = \int_0^L\varphi(x)^{2}\left|-\abs{u(x)}^{2}v(x)+(u(x)\cdot v(x))u(x)\right|^{2}dx }\\
			\vs
			&\ds{\quad  =\int_0^L\varphi(x)^{2}\big(\abs{u(x)}^{4}\abs{v(x)}^{2}-\abs{u(x)}^{2}(u(x)\cdot v(x))^{2}\big)dx   \leq c\,\vert\varphi\vert_{\infty}^{2}\abs{u}_{H^{1}}^{4}\abs{v}_{H}^{2}\leq c\vert\varphi\vert_{\infty}^{2}\abs{z}_{\H}^{6}}.
		\end{array}
	\end{equation*}
Moreover, 	\begin{equation*}
		\begin{array}{ll}
			&\ds{ \big\lvert \text{tr}_{K}(u_{1}\times (u_{1}\times v_{1}))-\text{tr}_{K}(u_{2}\times (u_{2}\times v_{2})) \big\rvert_{H}^{2}  }\\
			\vs
			&\ds{ \quad \leq c\int_0^L\varphi(x)^{2}\left(\big(\abs{u_{1}(x)}^{2}+\abs{u_{2}(x)}^{2}\big)\big(\abs{v_{1}(x)}^{2}+\abs{v_{2}(x)}^{2}\big)\abs{u_{1}(x)-u_{2}(x)}^{2}\right.}\\
			\vs
			&\ds{\quad \quad  \quad  \quad  \quad \quad \quad  \quad      \left.+\big(\abs{u_{1}(x)}^{4}+\abs{u_{2}(x)}^{4}\big)\abs{v_{1}(x)-v_{2}(x)}^{2}\right)dx  }\\
			\vs
			&\ds{\leq c\vert\varphi\vert_{\infty}^{2}\Big(\big(\abs{u_{1}}_{H^{1}}^{2}+\abs{u_{2}}_{H^{1}}^{2}\big)\big(\abs{v_{1}}_{H}^{2}+\abs{v_{2}}_{H}^{2}\big)\abs{u_{1}-u_{2}}_{H^{1}}^{2} + \big(\abs{u_{1}}_{H^{1}}^{4}+\abs{u_{2}}_{H^{1}}^{4}\big)\abs{v_{1}-v_{2}}_{H}^{2}   \Big) }\\
			\vs
			&\ds{\quad \quad \quad\quad\leq c\vert\varphi\vert_{\infty}^{2}\big(\abs{z_{1}}_{\H}^{4}+\abs{z_{2}}_{\H}^{4}\big)\abs{z_{1}-z_{2}}_{\H}^{2} }.
		\end{array}
	\end{equation*}
	This implies  that the mapping $z \in\,\mathcal{H}\mapsto \text{tr}_{K}\left[\partial_{v}\sigma(z)\sigma(z)\right] \in\,H$ is well-defined and locally Lischitz-continuous, with cubic growth.

{\em Case $k=1$.} For any  $z_{1}=(u_{1},v_{1})$ and  $z_{2}=(u_{2},v_{2})$ in $\H_{1}$, we have
	\[
		\begin{array}{l}
			\ds{ \norm{\sigma(z_{1})-\sigma(z_{2})}_{\mathcal{T}_{2}(K,H^{1})}^{2} = \sum_{i=1}^{\infty}\vert (\sigma(z_1)-\sigma(z_2))\xi_i\vert_{H^1}^2}\\
			\vs
			\ds{\quad \quad \quad \leq  2\sum_{i=1}^{\infty}
			\big\lvert \xi_{i}\,(u_{1}\times v_{1}-u_{2}\times v_{2})^\prime \big\rvert_{H}^{2}+2\sum_{i=1}^{\infty}\big\lvert\xi_{i}'\,(u_{1}\times v_{1}-u_{2}\times v_{2}) \big\rvert_{H}^{2}  }\\
			\vs
			\ds{\quad \leq c\vert\varphi\vert_{\infty}\Big(\abs{(u_{1}-u_{2})^\prime\times v_{1}}_{H}^{2} + \abs{(u_{1}-u_{2})\times v_{1}^\prime}_{H}^{2}+ \abs{u_{2}^\prime\times (v_{1}-v_{2}) }_{H}^{2} + \abs{u_{2}\times (v_{1}-v_{2})^\prime  }_{H}^{2}\Big)  }\\
			\vs
			\ds{ \quad\quad \quad \quad +c \vert\varphi_{1}\vert_{\infty}\Big(\abs{(u_{1}-u_{2})\times v_{1}}_{H}^{2}+\abs{u_{2}\times (v_{1}-v_{2})}_{H}^{2}\Big)  }\\
			\vs
			\ds{\quad \leq c\big(\vert\varphi\vert_{\infty}+\vert\varphi_{1}\vert_{\infty}\big)\big(\abs{u_{1}-u_{2}}_{H^{1}}^{2}\abs{v_{1}}_{H^{1}}^{2}+\abs{u_{2}}_{H^{1}}^{2}\abs{v_{1}-v_{2}}_{H^{1}}^{2}\big) }\\
			\vs
			\ds{\quad \quad \quad\quad \quad  \quad\quad\quad \leq c\big(\vert\varphi\vert_{\infty}+\vert\varphi_{1}\vert_{\infty}\big)\big(\abs{z_{1}}_{\H_{1}}^{2}+\abs{z_{2}}_{\H_{1}}^{2}\big)\abs{z_{1}-z_{2}}_{\H_{1}}^{2} }.
		\end{array}
	\] This implies the  local Lipschitz-continuity and polynomial growth of the mapping $\sigma:\mathcal{H}_1\to \mathcal{T}_2(K,H^1)$. 
	Finally, for every $z_{1}=(u_{1},v_{1})$ and  $z_{2}=(u_{2},v_{2})$ in $\H_{1}$
	\begin{equation*}
		\begin{array}{ll}
			&\ds{ \big\lvert \text{tr}_{K}(u_{1}\times (u_{1}\times v_{1}))-\text{tr}_{K}(u_{2}\times (u_{2}\times v_{2}))\big\rvert_{H^{1}}^{2}  }\\
			\vs
			&\ds{\quad \leq c\,\vert\varphi\vert_{\infty}^{2}\Big(\abs{(u_{1}-u_{2})^\prime\times (u_{1}\times v_{1})}_{H}^{2}+\abs{(u_{1}-u_{2})\times (u_{1}\times v_{1})^\prime}_{H}^{2}\Big) }\\
			\vs
			&\ds{\quad\quad +c\,\vert\varphi\vert_{\infty}^{2}\Big(\abs{u_{2}^\prime\times (u_{1}\times v_{1}-u_{2}\times v_{2})}_{H}^{2} + \abs{u_{2}\times (u_{1}\times v_{1}-u_{2}\times v_{2})^\prime}_{H}^{2}\Big) }\\
			\vs
			&\ds{\quad\quad +c\,\vert\varphi'\vert_{\infty}^{2}\vert u_{1}\times (u_{1}\times v_{1})-u_{2}\times (u_{2}\times v_{2})\vert_{H}^{2}  }\\
			\vs
			&\ds{\quad \quad\quad \quad\quad \quad \leq c\,\big(\vert \varphi\vert_\infty^{2}+\vert \varphi^\prime\vert_\infty^{2}\big)\big(\abs{z_{1}}_{\H_{1}}^{2}+\abs{z_{2}}_{\H_{1}}^{2}\big)^{2}\abs{z_{1}-z_{2}}_{\H_{1}}^{2} },
		\end{array}
	\end{equation*}
	and this implies that the mapping 
	\[\mathcal{H}_1 \ni \mapsto \text{tr}_K\left[u\times(u\times v)\right] \in\,H^1,\]
	is locally Lipschitz continuous and has cubic growth.

\end{proof}

 As a consequence of \eqref{sm1300}, \eqref{eqn-tr K-G'G} and \eqref{sm1310}, we can rewrite system  \eqref{SPDE1} as
 \begin{equation}\label{SPDE1-bis}
	\le\{\begin{array}{l}
		\ds{du_{\mu}(t)=v_{\mu}(t)dt, }\\
		\vs 
		\ds{dv_{\mu}(t)=\frac{1}{\mu}\bigl[\partial^2_x u_{\mu}(t)+\abs{\partial_x u_{\mu}(t)}_{H}^{2}u_{\mu}(t)-\mu\abs{v_{\mu}(t)}_{H}^{2}u_{\mu}(t)-\gamma v_{\mu}(t)}
		\\
		\vs
		\ds{ \quad \quad \quad \quad \quad +\frac 1{2\mu}\,\text{tr}_K(u_\mu(t)\times(u_\mu(t)\times v_\mu(t)))\bigr]\,dt +\frac{1}{\sqrt{\mu}}\big(u_{\mu}(t)\times v_{\mu}(t)\big) dw(t),  }\\
		\vs
		\ds{u_{\mu}(0)=u_{0},\ \ \ v_{\mu}(0)=v_{0},\ \ \  \ \ \ \ \ \ \ \ u_{\mu}(t,0)=u_{\mu}(t,L)=0 }.
	\end{array}\r.
\end{equation}
  
In particular, equation 
\eqref{SPDE2} can be rewritten as
\begin{equation*}
	\begin{array}{l}
		\ds{dz_\mu(t) = \mathcal{A}_\mu z_\mu(t)dt+\frac 1\mu\left(0, F_\mu(z_\mu(t))-\gamma v_\mu(t)\right)dt}\\
		\vs
		\ds{\quad \quad \quad \quad \quad \quad  \quad  +\frac{1}{2\mu}\left(0,\text{tr}_K\left(u_\mu\times(u_\mu(t)\times v_\mu(t))\right)\right)dt+\frac 1{\sqrt{\mu}}\left(0, u_\mu(t)\times v_\mu(t)\right) dw(t) },
	\end{array}
\end{equation*}
with $z_{\mu}(0)=z_0=(u_0,v_0)$.
This allows to say that  if a process $u_{\mu}(t)$ has $M$-valued trajectories of class $C^1$ and 
		\begin{equation*}
			v_{\mu}(t)=\partial_{t}u_{\mu}(t),\ \ \ \ t\geq0,
		\end{equation*}
and $z_\mu(t)=(u_\mu(t),v_\mu(t))$, then the process $z_{\mu}$ is a mild solution of equation \eqref{SPDE1}, with initial condition $z_0$, if for every $t\geq 0$, $\P$-almost surely,
		\begin{equation*}
			\begin{array}{l}
				\ds{z_{\mu}(t)=\mathcal{S}_{\mu}(t)z_0+\frac{1}{\mu}\int_{0}^{t}\mathcal{S}_{\mu}(t-s)\big(0,F_\mu(z_\mu(s))-\gamma v_{\mu}(s)\big)ds }\\
				\vs
				\ds{\quad \quad \quad \quad \quad+\frac{1}{2\mu}\int_{0}^{t}\mathcal{S}_{\mu}(t-s)\big(0,\text{{\em tr}}_K\left(u_{\mu}(s)\times (u_{\mu}(s)\times v_{\mu}(s))\right)\big)ds}\\
				\vs
				\ds{\quad\quad \quad \quad\quad \quad\quad \quad \quad+\frac{1}{\sqrt{\mu}}\int_{0}^{t}\mathcal{S}_{\mu}(t-s)\big(0,(u_{\mu}(s)\times v_{\mu}(s))\,dw(s)\big).  }
			\end{array}
		\end{equation*}

\subsection{Proof of Theorem \ref{well_posedness_system}}

It is immediate to check 
that $F_\mu:\mathcal{H}_{k}\to H^{k}$ is Lipschitz continuous when restricted to balls, for all $k\geq 0$,  and has cubic growth. Thus, in view of Lemma \ref{lem-sm1}, the proof follows from a  modification of the arguments introduced in  \cite[proof of Theorems 2.9 and 2.10]{BC23}.

Due to the local Lipschitz continuity of all coefficients in $\mathcal{H}$, equation \eqref{SPDE2} admits a unique maximal local mild solution $z_\mu \in\,C([0,\tau_\mu);\mathcal{H})$, defined up to a certain stopping time $\tau_\mu$.
Our purpose is showing that $z_\mu(t)\in\, \mathcal{M}$, for all $t \in\,[0,\tau_\mu)$, and 
\begin{equation}\label{sm4}
\P(\tau_\mu=\infty)=1.
\end{equation}
In this way, we get the existence and uniqueness of a global mild solution $z_\mu \in\,C([0,+\infty);\mathcal{H}\cap \mathcal{M})$.

In order to prove the invariance of the tangent bundle $\mathcal{M}$, we introduce the following processes 
\[\vartheta_\mu(t):=\frac{1}{2}\left( \vert u_\mu(t)\vert_{H}^2-1\right), \ \ \ \ \ \ \ \eta_\mu(t):=\langle u_\mu(t),v_\mu(t)\rangle_H, \ \ \ \ t\in [0,\tau_\mu).\]
If we show that they satisfy the linear system
\begin{equation}
\label{eq-3.13-psi}
\left\{\begin{split}
d\vartheta_\mu(t)&=\eta_\mu(t)\,dt \\[10pt]
d\eta_\mu(t)&+\frac{\gamma}{\mu}\,d\vartheta_\mu(t)=\left(\frac 1\mu \vert u_\mu(t)\vert_{H^1}^2-\vert v_\mu(t)\vert_H^2\right)\vartheta_\mu(t),\ \ \ \ \ \ \ \ t\in [0, \tau_\mu),
\end{split}\right.
\end{equation}
 since   $\vartheta_\mu(0)=\eta_\mu(0)=0$, we obtain that   $\vartheta_\mu(t)=\eta_\mu(t)= 0$, for every $t \in\,[0,\tau_\mu)$, $\mathbb{P}$-a.s., and  this implies that $z_\mu(t) \in\,\mathcal{M}$, for every $t \in\,[0,\tau_\mu)$, $\mathbb{P}$-a.s.
 
 As in \cite{BC23}, it can be shown the following fact.

\begin{Lemma}\label{lem-Ito}
Assume that  a local process $z_\mu(t)=(u_\mu(t),v_\mu(t))$, $t \in\,[0,\sigma_\mu)$ is a solution to
\begin{equation}   \label{eq-3.10}
 z_\mu(t)=\mathcal{S}_\mu(t)z_0+ \int_0^t \mathcal{S}_\mu(t-s) \bigl(0,  f(s)
   \bigr) \, ds
+ \int_0^t \mathcal{S}_\mu(t-s) \bigl(0, g(s) \bigr)\,dw(s), \ \ \ \ \ \ \ t\in\,[0,\sigma_\mu).
\end{equation}
where all processes are progressively measurable,  $f$ is $H$-valued and $g$ is $\mathcal{T}_2(K,H)$-valued. Then, for every $t  \in\,[0,\sigma_\mu)$, $\mathbb{P}$-almost surely,
\begin{equation}
\begin{split}
\langle u_\mu(t),v_\mu(t)\rangle_H &=\langle u_\mu(0),v_\mu(0)\rangle_H -\frac 1\mu\int_0^t  \vert \partial_x u_\mu(s)\vert_{H}^2\,ds+\int_0^t  \langle u_\mu(s),f(s)\rangle_H\, ds
\\[10pt] 
&+\int_0^t  \vert v_\mu(s)\vert_{H}^2\,ds+\int_0^t  \langle u_\mu(s),g(s)\,dw(s)\rangle_H,
\end{split}
\label{eqn-3.11}
\end{equation}
and
\begin{equation}
\label{eqn-3.12}	
\begin{array}{l}
\ds{\vert v_\mu(t)\vert_H^2+\frac 2\mu \vert\partial_x u_\mu(t)\vert_H^2=\vert v_\mu(0)\vert_H^2+\frac 2\mu \vert\partial_x u_\mu(0)\vert_H^2+2\int_0^t\langle v_\mu(s),f(s)\rangle_H\,ds}\\
\vs 
\ds{\quad \quad  \quad  \quad  +2\int_0^t\langle v_\mu(s),g(s)dw(s)\rangle_H+\int_0^t\Vert g(s)\Vert_{\mathcal{T}_2(K,H)}^2\,ds.}
\end{array}
\end{equation}
\hfill $\Box$
\end{Lemma}

Now, the local solution $z_\mu(t)=(u_\mu(t), v_\mu(t))$ of equation \eqref{SPDE2} satisfies equation \eqref{eq-3.10}, with
\[f(s):= -\vert v_\mu(s)\vert_H^2u_\mu(s) +\frac 1\mu\vert u_\mu(s)\vert_{H^1}^2u_\mu(s)-\frac{\gamma}{\mu} v_\mu(s) +\frac1{2\mu} \text{tr}_{K} \left[u_\mu(s)\times(u_\mu(s)\times v_\mu(s))\right],\]
and
\[g(s):=\frac 1{\sqrt{\mu}}\sigma(z_\mu(s))=\frac 1{\sqrt{\mu}}\,u_\mu(s)\times v_\mu(s).\]
Notice that
\[\langle u_\mu(s),f(s)\rangle_H= -\vert v_\mu(s)\vert_H^2\vert u_\mu(s)\vert_H^2 +\frac 1\mu\vert u_\mu(s)\vert_{H^1}^2\vert u_\mu(s)\vert_H^2-\frac{\gamma}{\mu} \eta_\mu(s),\]
and for every $\xi \in\,K$
\[\langle u_\mu(s),g(s)\xi\rangle_H=0.\]

Thus, thanks to identity \eqref{eqn-3.11} in Lemma \ref{lem-Ito}, we have
\[\begin{array}{l}
\ds{\eta_\mu(t)-\eta_\mu(0)}\\
\vs 
\ds{=
\frac 1\mu\int_0^t  \vert  u_\mu(s)\vert_{H^1}^2\left(\vert u_\mu(s)\vert_H^2-1\right)\,ds-\int_0^t  \vert v_\mu(s)\vert_H^2\left(\vert u_\mu(s)\vert_H^2-1\right)\,ds-\frac \gamma\mu\int_0^t\langle u_\mu(s),v_\mu(s)\rangle_H\,ds}\\
\vs
\ds{=\frac 1\mu\int_0^t  \vert  u_\mu(s)\vert_{H^1}^2\,\vartheta_\mu(s)\,ds-\int_0^t  \vert v_\mu(s)\vert_H^2\,\vartheta_\mu(s)\,ds-\frac \gamma\mu\int_0^t\eta_\mu(s)\,ds.}	
\end{array}\]
In particular, the processes $\vartheta_\mu(t)$ and $\eta_\mu(t)$ satisfy  equation \eqref{eq-3.13-psi}, and, as explained above, this implies that $z_\mu(t) \in\,\mathcal{M}$, for every $t \in\,[0,\tau_\mu)$, $\mathbb{P}$-a.s.

Next, let us prove \eqref{sm4}. Since $(u_\mu(t),v_\mu(t)) \in\,\mathcal{M}$, we have
\[\langle v_\mu(t),f(t)\rangle_H=-\frac\gamma\mu\,\vert v_\mu(t)\vert_H^2+\frac1{2\mu} \langle v_\mu(t),\text{tr}_{K} \left[u_\mu(s)\times(u_\mu(s)\times v_\mu(s))\right]\rangle_H.\]
As
\[ \langle v_\mu(t),g(t)\xi\rangle_H=0,\ \ \ \ \ \  \xi \in\,K,\]  for every $t<\tau_\mu$ this gives
\begin{equation*}
	\begin{array}{ll}
		&\ds{d\abs{v_\mu(t)}_{H}^{2}=\frac{2}{\mu}\Inner{v_\mu(t),\partial^2_x u_\mu(t)}_{H}dt-\frac{2\gamma}\mu\abs{v_\mu(t)}_{H}^{2}dt+\frac 1\mu\,\norm{u(t)\times v(t)}_{\mathcal{T}_{2}(K,H)}^{2}dt}\\
		\vs
		&\ds{\quad\quad\quad\quad \quad \quad +\frac1{\mu} \langle v_\mu(t),\text{tr}_{K} \left[u_\mu(s)\times(u_\mu(s)\times v_\mu(s))\right]\rangle_H\,dt.}
	\end{array}
\end{equation*}	
In view of \eqref{sm20},
for every $h,k \in\,\mathbb{R}^3$ we have
\[\vert h\times k\vert^2+(k\cdot[h\times(h\times k)])=\vert h\times k\vert^2+\vert(h\cdot k)\vert^2-\vert h\vert^2\vert k\vert^2=0.\]
Therefore, we obtain
\[d\abs{v_\mu(t)}_{H}^{2}=	-\frac 1\mu\,d\abs{u_\mu(t)}_{H^{1}}^{2}-\frac{2\gamma}\mu\abs{v_\mu(t)}_{H}^{2}dt,  \]
and this implies that for every $t<\tau_\mu$ 
\begin{equation}\label{sm10-bis}
	\abs{u_\mu(t)}_{H^{1}}^{2}+\mu\,\abs{v_\mu(t)}_{H}^{2}+2\gamma\int_{0}^{t}\abs{v_\mu(s)}_{H}^{2}ds=\abs{u_0}_{H^{1}}^{2}+\mu\,\abs{v_0}_{H}^{2},\ \ \ \ \P\text{-a.s.}
\end{equation}
In particular, we can  conclude that
for every $\mu>0$  there exists some deterministic constant $\kappa_\mu>0$ such that
\[\sup_{t \in\,[0,\tau_\mu)}\vert z_\mu(t)\vert_{\mathcal{H}}\leq \kappa_\mu,\ \ \ \ \ \ \mathbb{P}-\text{a.s.},\]
and  \eqref{sm4} follows. Finally, by combining together  \eqref{sm4} and \eqref{sm10-bis}, we obtain \eqref{sm2}.


\section{Energy Estimates}
\label{sec5}

In Theorem \ref{well_posedness_system} we have seen that for every  $(u_0,v_0)\in \H\cap \M$ and  $\mu>0$, equation \eqref{SPDE1} (and, equivalently, equation \eqref{SPDE1-bis}) admits a unique mild solution  $z_{\mu}=(u_{\mu},v_{\mu})\in C([0,+\infty);\H)$. Moreover, for every $\mu>0$ and $t>0$ the following identity holds
	\begin{equation}\label{uniform_est1}
	\abs{u_{\mu}(t)}_{H^{1}}^{2}+\mu\abs{v_{\mu}(t)}_{H}^{2}+2\gamma\int_{0}^{t}\abs{v_{\mu}(s)}_{H}^{2}ds=\abs{u_0}_{H^{1}}^{2}+\mu\abs{v_0}_{H}^{2},\ \ \ \ \P\text{-a.s.}
	\end{equation}
In this section, our purpose is proving that for every $(u_0,v_0)\in \H_{1}\cap\M$, there exists a constant $c>0$  such that for every $\mu\in(0,1)$ and $t\geq 0$
	\begin{equation}\label{uniform_est4-bis}
		\E\sup_{r\in[0,t]}\Big(\abs{u_{\mu}(r)}_{H^{2}}^{2}+\mu\abs{v_{\mu}(r)}_{H^{1}}^{2}\Big)+\E\int_{0}^{t}\abs{v_{\mu}(s)}_{H^{1}}^{2}ds\leq c.
	\end{equation}
The inequality above is a consequence of the following Lemma.

\begin{Lemma} \label{lemma3.1}
	For every $(u_0,v_0)\in \H_{1}\cap \M$, there exist two constants $c_1, c_2>0$ depending only on $\vert(u_0,\sqrt{\mu}\,v_0)\vert_{\mathcal{H}_1}$,  $\varphi$ and $\varphi_{1}$, such that for every $\mu\in(0,1)$ and $t\geq 0$
	\begin{equation}\label{uniform_est2}
		\begin{array}{l}
			\ds{\E\sup_{r \in\,[0,t]}\Big(\abs{u_{\mu}(r)}_{H^{2}}^{2}+\mu\abs{v_{\mu}(r)}_{H^{1}}^{2}+\mu\abs{u_{\mu}(r)}_{H^{1}}^{2}\abs{v_{\mu}(r)}_{H}^{2}\Big) }\\
			\vs
			\ds{\quad +\E\int_{0}^{t}\Big(\abs{v_{\mu}(s)}_{H^{1}}^{2}+ \abs{u_{\mu}(s)}_{H^{2}}^{2}\abs{v_{\mu}(s)}_{H}^{2}+\mu\abs{v_{\mu}(s)}_{H^{1}}^{2}\abs{v_{\mu}(s)}_{H}^{2}+\mu\abs{u_{\mu}(s)}_{H^{1}}^{2}\abs{v_{\mu}(s)}_{H}^{4}\Big)ds}\\
			\vs 
			\ds{\quad \quad \quad \quad \quad \leq c_1 \Big(\abs{u_0}_{H^{2}}^{2}+\mu\abs{v_0}_{H^{1}}^{2}+\mu\abs{u_{0}}_{H^{1}}^{2}\abs{v_0}_{H}^{2}\Big)\exp\Big(c_2\big(\abs{u_0}_{H^{1}}^{2}+\mu\abs{v_0}_{H}^{2}\big)\Big) }.
		\end{array}
	\end{equation}
\end{Lemma}

\begin{proof}
By applying  It\^o's formula to $\abs{v_{\mu}(t)}_{H^{1}}^{2}$, we get 
\begin{equation}\label{sm9}
	\begin{array}{ll}
		&\ds{d\abs{v_{\mu}(t)}_{H^{1}}^{2} = \frac{2}{\mu}\Big(\Inner{v_{\mu}(t),\partial^2_x u_{\mu}(t)}_{H^{1}}+\abs{u_{\mu}(t)}_{H^{1}}^{2}\Inner{v_{\mu}(t),u_{\mu}(t)}_{H^{1}}-\mu\abs{v_{\mu}(t)}_{H}^{2}\Inner{v_{\mu}(t),u_{\mu}(t)}_{H^{1}}}\\
		\vs
		&\ds{-\gamma\abs{v_{\mu}(t)}_{H^{1}}^{2}+\frac{1}{2}\Inner{v_{\mu}(t),\text{tr}_{K}\big(u_{\mu}(t)\times (u_{\mu}(t)\times v_{\mu}(t))\big) }_{H^{1}}+\frac{1}{2}\norm{u_{\mu}(t)\times v_{\mu}(t)}_{\mathcal{T}_{2}(K,H^{1})}^{2}\Big)\,dt }\\
		\vs
		&\ds{\quad\quad\quad\quad \quad\quad\quad\quad+\frac{2}{\sqrt{\mu}}\Inner{v_{\mu}(t), (u_{\mu}(t)\times v_{\mu}(t))dw(t) }_{H^{1}}  }\\
		\vs
		&\ds{\quad\quad\quad\quad =-\frac{1}{\mu}d\abs{u_{\mu}(t)}_{H^{2}}^{2}+\abs{u_{\mu}(t)}_{H^{1}}^{2}\Big(\frac{1}{\mu}d\abs{u_{\mu}(t)}_{H^{1}}^{2}+d\abs{v_{\mu}(t)}_{H}^{2}\Big)-d\Big(\abs{u_{\mu}(t)}_{H^{1}}^{2}\abs{v_{\mu}(t)}_{H}^{2}\Big)  }\\
		\vs
		&\ds{+\frac{1}{\mu}\Big(-2\gamma\abs{v_{\mu}(t)}_{H^{1}}^{2}+\Inner{v_{\mu}(t),\text{tr}_{K}\big(u_{\mu}(t)\times (u_{\mu}(t)\times v_{\mu}(t))\big) }_{H^{1}}+\norm{u_{\mu}(t)\times v_{\mu}(t)}_{\mathcal{T}_{2}(K,H^{1})}^{2}\Big)\,dt }\\
		\vs
		&\ds{\quad\quad\quad\quad \quad\quad\quad\quad\quad+\frac{2}{\sqrt{\mu}}\Inner{v_{\mu}(t), (u_{\mu}(t)\times v_{\mu}(t))dw(t) }_{H^{1}}  }.
	\end{array}
\end{equation}
According to \eqref{uniform_est1}, this implies that
\begin{equation*}
	\begin{array}{l}
		\ds{ d\Big(\abs{u_{\mu}(t)}_{H^{2}}^{2}+\mu\abs{v_{\mu}(t)}_{H^{1}}^{2}+\mu\abs{u_{\mu}(t)}_{H^{1}}^{2}\abs{v_{\mu}(t)}_{H}^{2}  \Big)    }\\
		\vs
		\ds{\quad\quad =-2\gamma\abs{u_{\mu}(t)}_{H^{1}}^{2}\abs{v_{\mu}(t)}_{H}^{2}dt-2\gamma\abs{v_{\mu}(t)}_{H^{1}}^{2}dt+\Inner{v_{\mu}(t),\text{tr}_{K}\big(u_{\mu}(t)\times(u_{\mu}(t)\times v_{\mu}(t))\big) }_{H^{1}}dt   }\\
		\vs
		\ds{\quad\quad\quad\quad +\norm{u_{\mu}(t)\times v_{\mu}(t)}_{\mathcal{T}_{2}(K,H^{1})}^{2}dt+2\sqrt{\mu} \Inner{v_{\mu}(t),(u_{\mu}(t)\times v_{\mu}(t))dw(t) }_{H^{1}} }.
	\end{array}
\end{equation*}
Now, let us define 
\begin{equation}\label{sm10}
	Y_{\mu,a}(t):=\exp\left(-a\int_{0}^{t} \abs{v_{\mu}(s)}_{H}^{2}ds\right),\ \ \ \ t\geq 0,\ \ \ \ \mu>0,
\end{equation}
for some constant $a>0$ to be determined later. We have
\begin{equation*}
	\begin{array}{l}
		\ds{d\Big(Y_{\mu,a}(t)\Big(\abs{u_{\mu}(t)}_{H^{2}}^{2}+\mu\abs{v_{\mu}(t)}_{H^{1}}^{2}+\mu\abs{u_{\mu}(t)}_{H^{1}}^{2}\abs{v_{\mu}(t)}_{H}^{2}  \Big) \Big)   }\\
		\vs
		\ds{\quad =Y_{\mu,a}(t)\Big( -a\abs{u_{\mu}(t)}_{H^{2}}^{2}\abs{v_{\mu}(t)}_{H}^{2}-a\mu\abs{v_{\mu}(t)}_{H^{1}}^{2}\abs{v_{\mu}(t)}_{H}^{2}-a\mu\abs{u_{\mu}(t)}_{H^{1}}^{2}\abs{v_{\mu}(t)}_{H}^{4}  }\\
		\vs
		\ds{\quad\quad\quad -2\gamma\abs{u_{\mu}(t)}_{H^{1}}^{2}\abs{v_{\mu}(t)}_{H}^{2}-2\gamma\abs{v_{\mu}(t)}_{H^{1}}^{2}+\Inner{v_{\mu}(t),\text{tr}_{K}\big(u_{\mu}(t)\times(u_{\mu}(t)\times v_{\mu}(t))\big) }_{H^{1}}   }\\
		\vs
		\ds{\quad\quad\quad\quad \quad+\norm{u_{\mu}(t)\times v_{\mu}(t)}_{\mathcal{T}_{2}(K,H^{1})}^{2}\Big) dt +2\sqrt{\mu}\ Y_{\mu,a}(t)\Inner{v_{\mu}(t),(u_{\mu}(t)\times v_{\mu}(t))dw(t) }_{H^{1}} }.
	\end{array}
\end{equation*}
Thus, if we integrate with respect to $t\geq 0$, we get 
\begin{equation}\label{sm102}
	\begin{array}{l}
		\ds{ Y_{\mu,a}(t)\Big( \abs{u_{\mu}(t)}_{H^{2}}^{2}+\mu\abs{v_{\mu}(t)}_{H^{1}}^{2}+\mu\abs{u_{\mu}(t)}_{H^{1}}^{2}\abs{v_{\mu}(t)}_{H}^{2} \Big) +\int_{0}^{t}Y_{\mu,a}(s)\Big( a\abs{u_{\mu}(s)}_{H^{2}}^{2}\abs{v_{\mu}(s)}_{H}^{2} }\\
		\vs
		\ds{\quad +a\mu\abs{v_{\mu}(s)}_{H^{1}}^{2}\abs{v_{\mu}(s)}_{H}^{2}+a\mu\abs{u_{\mu}(s)}_{H^{1}}^{2}\abs{v_{\mu}(s)}_{H}^{4}+2\gamma\abs{u_{\mu}(s)}_{H^{1}}^{2}\abs{v_{\mu}(s)}_{H}^{2}+2\gamma\abs{v_{\mu}(s)}_{H^{1}}^{2}\Big)ds }\\
		\vs
		\ds{ \quad \quad\quad \quad =\abs{u_0}_{H^{2}}^{2}+\mu\abs{v_0}_{H^{1}}^{2}+\mu\abs{u_0}_{H^{1}}^{2}\abs{v_0}_{H}^{2} }\\
		\vs
		\ds{\quad \quad \quad \quad \quad + \int_{0}^{t}Y_{\mu,a}(s)\,J(u_{\mu}(s),v_{\mu}(s))ds+2\sqrt{\mu}\int_{0}^{t}Y_{\mu,a}(s)\,G(u_{\mu}(s),v_{\mu}(s))dw(s)  },
	\end{array}
\end{equation} 
where 
\begin{equation*}
	J(u,v):=\Inner{v,\text{tr}_{K}(u\times (u\times v))}_{H^{1}}+\norm{u\times v}_{\mathcal{T}_{2}(K,H^{1})}^{2},\ \ \ \ (u,v)\in\H_{1},
\end{equation*}
and 
\begin{equation*}
	G(u,v)\xi:=\Inner{v,(u\times v)\xi}_{H^{1}},\ \ \ \ (u,v)\in\H_{1},\ \ \ \ \xi\in K.
\end{equation*}
Note that for every $(u,v)\in \H_{1}$
\begin{equation*}
	\begin{array}{l}
		\ds{D\big(\text{tr}_{K}\big(u\times (u\times v)\big)\big) =\left(-\abs{u}^{2}v+(u\cdot v)u\right)\varphi' }\\
		\vs
		\ds{\quad  \quad \quad \quad + \Big(-2(Du\cdot u)v-\abs{u}^{2}Dv+(Du\cdot v)u+(u\cdot Dv)u+(u\cdot v)Du\Big)\varphi },
	\end{array}
\end{equation*}
so that we have 
\begin{equation}\label{sm32}
	\begin{array}{ll}
		&\ds{\Inner{v,\text{tr}_{K}\big(u\times (u\times v)\big) }_{H^{1}} = \int_0^L\Big(-\abs{u}^{2}(v\cdot Dv)+(u\cdot v)(u\cdot Dv)\Big)\varphi'dx  }\\
		\vs
		&\ds{ +\int_0^L\Big(-2(Du\cdot u)(Dv\cdot v)-\abs{u}^{2}\abs{Dv}^{2}+(Du\cdot v)(u\cdot Dv)+(u\cdot Dv)^{2}+(u\cdot v)(Du\cdot Dv) \Big)\varphi dx }.
	\end{array}
\end{equation}
Moreover, 
\begin{equation} \label{sm33}
	\begin{array}{l}
		\ds{\norm{u\times v}_{\mathcal{T}_{2}(K,H^{1})}^{2} =\sum_{i=1}^{\infty}\int_0^L\big\lvert (u\times v)\xi_{i}'+D(u\times v)\xi_{i} \big\rvert^{2}dx=\int_0^L\sum_{i=1}^{\infty}\big\lvert (u\times v)\xi_{i}'+D(u\times v)\xi_{i} \big\rvert^{2}dx }\\
		\vs
		\ds{\quad\quad=\int_0^L|u\times v|^{2} \varphi_{1} dx + \int_0^L\big\lvert D(u\times v)\big\rvert^{2}  \varphi dx + 2\int_0^L(u\times v\cdot D(u\times v))\sum_{i}\xi_{i}\xi_{i}'dx }\\
		\vs
		\ds{\quad\quad = \int_0^L|u\times v|^{2} \varphi_{1}dx+\int_0^L\big|Du\times v+u\times Dv\big|^{2}\varphi dx }\\
		\vs 
		\ds{\quad\quad  \quad\quad \quad\quad \quad\quad +2\int_0^L(u\times v)\cdot \big(Du\times v+u\times Dv\big)\sum_{i}\xi_{i}\xi_{i}'dx  }.
	\end{array}
\end{equation}
Due to the well-known identity
\[((a\times b)\cdot(c\times d))=(a\cdot c)(b\cdot d)-(a\cdot d)(b\cdot c),\ \ \ \ \ \ \ a, b, c, d \in\,\mathbb{R}^3,\]
from \eqref{sm32} and \eqref{sm33} we get
\begin{equation}\label{sm11}
	\begin{array}{l}
	\ds{J(u,v)=\int_0^L(u\times v)^{2}\varphi_{1}dx+2\int_0^L\big((u\times v)\cdot (Du\times v)\big)\sum_{i}\xi_{i}\xi_{i}'\ dx}\\
	\vs
	\ds{\quad \quad \quad \quad \quad \quad \quad+\int_0^L\big[(Du\times v)^{2}-\big((Du\times u)\cdot (Dv\times v)\big)\big]\varphi \ dx.}\end{array}
\end{equation}
In particular, this implies that 
 for any $\epsilon>0$ 
\begin{equation*}
	\begin{array}{ll}
		&\ds{\abs{J(u,v)}\leq c\vert\varphi_{1}\vert_{\infty} \abs{u}_{H^{1}}^{2}\abs{v}_{H}^{2}+c\sqrt{\vert\varphi\vert_{\infty}\vert\varphi_{1}\vert_{\infty}} \abs{u}_{H^{1}}\abs{u}_{H^{2}}\abs{v}_{H}^{2} }\\
		\vs
		&\ds{\quad\quad \quad\quad \quad\quad +c\vert\varphi\vert_{\infty} \abs{u}_{H^{2}}\abs{u}_{H^{1}}\abs{v}_{H^{1}}\abs{v}_{H}\leq 
		\epsilon\,\abs{u}_{H^{1}}^2\abs{v}_{H^{1}}^{2}+c(\epsilon)\abs{u}_{H^{2}}^{2}\abs{v}_{H}^{2},}
	\end{array}
\end{equation*}
for some constant $c(\epsilon)=c(\epsilon,\vert\varphi\vert_{\infty},\vert\varphi_{1}\vert_{\infty})>0$.
Hence, according to \eqref{uniform_est1}, we obtain
\begin{equation}
\label{sm100}
\begin{array}{l}
\ds{	\mathbb{E}\sup_{r \in\,[0,t]}\,\left|\int_{0}^{r}Y_{\mu,a}(s)\,J(u_{\mu}(s),v_{\mu}(s))ds\right|}\\
\vs
\ds{\quad \quad \leq \epsilon\left(\abs{u_0}_{H^{1}}^{2}+\mu\abs{v_0}_{H}^{2}\right)\mathbb{E}\int_0^t	Y_{\mu,a}(s)\vert v_{\mu}(s)\vert_{H^1}^2\,ds+c(\epsilon)\,\mathbb{E}\int_0^t	Y_{\mu,a}(s)\vert u_\mu(s)\vert_{H^2}^2\vert v_{\mu}(s)\vert_{H}^2\,ds.}
\end{array}
\end{equation}

Next, for every $(u,v)\in \H_{1}$ and $k \in\,K$ we have
\[\begin{array}{l}
\ds{
G(u,v)k=\int_0^L\left(Dv\cdot (u\times v)k^\prime+\left[Du\times v+u\times Dv\right]k\right)\,dx}\\
\vs 
\ds{\quad \quad\quad \quad \quad\quad =	\int_0^L\left(Dv\cdot (u\times v)k^\prime+(Du\times v)k\right)\,dx,}	
\end{array}
\]
so that
\begin{equation}\label{sm13}
	\begin{array}{ll}
		&\ds{\norm{G(u,v)}_{\mathcal{T}_{2}(K,\mathbb{R})}^{2} = \sum_{i=1}^{\infty}\Big\lvert \int_0^L\left[\left(Dv\cdot(u\times v)\right)\xi_{i}'+\left(Dv\cdot(Du\times v)\right)\xi_{i}\right]\ dx\Big\rvert^{2} }\\
		\vs
		&\ds{\quad\quad \quad \quad \leq \abs{v}_{H^{1}}^{2}\sum_{i=1}^{\infty}\Big(\abs{(u\times v)\xi_{i}}_{H}^{2}+\abs{(Du\times v)\xi_{i}}_{H}^{2}\Big)\leq c\big(\vert \varphi\vert_\infty+\vert \varphi_1\vert_\infty\big)\abs{u}_{H^{2}}^{2}\abs{v}_{H^{1}}^{2}\abs{v}_{H}^{2} }.
	\end{array}
\end{equation}
This implies that for every $\epsilon>0$ we can fix  some $c(\epsilon)=c(\epsilon,\vert\varphi\vert_{\infty},\vert\varphi_{1}\vert_{\infty})>0$ such that,
\begin{equation}\label{sm101}
	\begin{array}{l}
		\ds{\sqrt{\mu}\ \E\sup_{r\in[0,t]}\Big\lvert \int_{0}^{r}Y_{\mu,a}(s)G(u_{\mu}(s),v_{\mu}(s))dw(s)\Big\rvert \leq c\,\E\left(\int_{0}^{t}\mu Y_{\mu,a}^{2}(s)\norm{G(u_{\mu}(s),v_{\mu}(s))}_{\mathcal{T}_{2}(K,\mathbb{R})}^{2}ds\right)^{\frac{1}{2}} }\\
		\vs
		\ds{\quad \quad \quad \quad \leq c\,\big(\vert \varphi\vert_\infty+\vert \varphi_1\vert_\infty\big)^{\frac{1}{2}}\E\left(\int_{0}^{t}\mu Y_{\mu,a}^{2}(s)\abs{u_{\mu}(s)}_{H^{2}}^{2}\abs{v_{\mu}(s)}_{H^{1}}^{2}\abs{v_{\mu}(s)}_{H}^{2}\,ds\right)^{\frac{1}{2}} }\\
		\vs
		\ds{\quad \quad \quad \quad \quad \quad \quad \quad \leq \epsilon\mu \ \E\sup_{r\in[0,t]}Y_{\mu,a}(r)\abs{v_{\mu}(r)}_{H^{1}}^{2}+c(\epsilon)\E\int_{0}^{t}Y_{\mu,a}(s)\abs{u_{\mu}(s)}_{H^{2}}^{2}\abs{v_{\mu}(s)}_{H}^{2}ds.  }
	\end{array}
\end{equation}

Therefore, if we pick \[\bar{\epsilon}:=\frac 12\wedge \gamma\,\left(\abs{u_0}_{H^{1}}^{2}+\mu\abs{v_0}_{H}^{2}\right)^{-1},\] 
 and
$\overline{a}=a(\bar{\epsilon})>0$ large enough  in \eqref{sm100} and \eqref{sm101}, from  \eqref{sm102} we get
\begin{equation*}
	\begin{array}{l}
		\ds{\E\sup_{r\in[0,t]}\Bigg(Y_{\mu,\overline{a}}(r)\Big(\abs{u_{\mu}(r)}_{H^{2}}^{2}+\mu\abs{v_{\mu}(r)}_{H^{1}}^{2}+\mu\abs{u_{\mu}(r)}_{H^{1}}^{2}\abs{v_{\mu}(r)}_{H}^{2}  \Big) \Bigg)   }\\
		\vs
		\ds{+c\,\E\int_{0}^{t}\ Y_{\mu,\overline{a}}(s)\Big(\abs{v_{\mu}(s)}_{H^{1}}^{2}+ \abs{u_{\mu}(s)}_{H^{2}}^{2}\abs{v_{\mu}(s)}_{H}^{2}+\mu\abs{v_{\mu}(s)}_{H^{1}}^{2}\abs{v_{\mu}(s)}_{H}^{2}+\mu\abs{u_{\mu}(s)}_{H^{1}}^{2}\abs{v_{\mu}(s)}_{H}^{4}\Big)ds  }\\
		\vs
		\ds{\quad\quad  \quad\quad  \quad\quad  \quad\quad  \quad\quad  \leq \abs{u_0}_{H^{2}}^{2}+\mu\abs{v_0}_{H^{1}}^{2}+\mu\abs{u_{0}}_{H^{1}}^{2}\abs{v_0}_{H}^{4}  }.
	\end{array}
\end{equation*}
Finally, since \eqref{uniform_est1} gives for every $t\geq0$
\begin{equation*}
	\int_{0}^{t}\abs{v_{\mu}(s)}_{H}^{2}ds\leq \frac{1}{2\gamma}\big(\abs{u_0}_{H^{1}}^{2}+\mu\abs{v_0}_{H}^{2}\big),\ \ \ \ \P-\text{a.s.},
\end{equation*}
we conclude 
\begin{equation*}
	\begin{array}{ll}
	&\ds{\E\sup_{r\in[0,t]}\Big(\abs{u_{\mu}(r)}_{H^{2}}^{2}+\mu\abs{v_{\mu}(r)}_{H^{1}}^{2}+\mu\abs{u_{\mu}(r)}_{H^{1}}^{2}\abs{v_{\mu}(r)}_{H}^{2}\Big) }\\
	\vs
	&\ds{\quad +c\E\int_{0}^{t}\Big(\abs{v_{\mu}(s)}_{H^{1}}^{2}+ \abs{u_{\mu}(s)}_{H^{2}}^{2}\abs{v_{\mu}(s)}_{H}^{2}+\mu\abs{v_{\mu}(s)}_{H^{1}}^{2}\abs{v_{\mu}(s)}_{H}^{2}+\mu\abs{u_{\mu}(s)}_{H^{1}}^{2}\abs{v_{\mu}(s)}_{H}^{4}\Big)ds  }\\
	\vs
	&\ds{\quad \quad \quad \quad \quad \leq \Big(\abs{u_0}_{H^{2}}^{2}+\mu\abs{v_0}_{H^{1}}^{2}+\mu\abs{u_{0}}_{H^{1}}^{2}\abs{v_0}_{H}^{2}\Big)\exp\Big(\frac{\overline{a}}{2\gamma}\big(\abs{u_0}_{H^{1}}^{2}+\mu\abs{v_0}_{H}^{2}\big)\Big)  },
	\end{array}
\end{equation*}
and this implies \eqref{uniform_est2}.

\end{proof}

\section{The limiting equation}\label{sec6}

We consider the following deterministic equation
\begin{equation}\label{limiting_equation}
	\le\{\begin{array}{l}
		\ds{\partial_{t}\Big[\Big(\gamma+\frac{1}{2}\varphi\abs{u(t)}^{2}\Big)u(t)\Big]=\partial^2_x u(t) +  \abs{\partial_x u(t)}_{H}^{2}u(t) +\frac{3\varphi}{2\gamma}\left(\left[\partial^2_x u(t)+\abs{\partial_x u(t)}_{H}^2u(t)\right] \cdot u(t)\right)u(t),}\\
		\vs 
		\ds{u(0,x)=u_{0}(x),\ \ \ \ u(t,0)=u(t,L)=0 },
	\end{array}\r.
\end{equation}
where, we recall,  $\varphi(x):=\sum_{i=1}^{\infty}\abs{\xi_{i}(x)}^{2}$, for $x\in (0,L)$.

\begin{Definition}
	Let $u_0\in H^{1}\cap M$. We say that $u$ is a solution to equation \eqref{limiting_equation} in $[0,T]$  if 
	\begin{equation*}
		u\in C([0,T];H^{1})\cap L^2(0,T;H^2),\ \ \ \ \ \ \  \partial_{t}u\in L^{2}(0,T;H),
	\end{equation*}
	and  the identity 
	\[\begin{array}{l}
		\ds{\Big(\gamma+\frac{1}{2}\varphi\abs{u(t)}^{2}\Big)u(t)=\Big(\gamma+\frac{1}{2}\varphi\abs{u_0}^{2}\Big)u_0}\\
		\vs 
		\ds{\quad \quad \quad  
+\int_0^t\left(\partial^2_x u(s) +  \abs{\partial_x u(s)}_{H}^{2}u(s) +\frac{3\varphi}{2\gamma}\left(\left[\partial^2_x u(s)+\abs{\partial_x u(s)}_{H}^2u(s)\right] \cdot u(s)\right)u(s)\right)\,ds},
	\end{array}\]
	holds in $H$, for a.e. $t \in\,[0,T]$.
	
\end{Definition}

In the following lemma we show that equation \eqref{limiting_equation} has an equivalent formulation.
\begin{lem}
	Let $u_0\in H^{1}\cap M$. Then any function $u \in\,C([0,T];H^{1})\cap L^2(0,T;H^2)$, with $\partial_{t}u\in L^{2}(0,T;H)$, satisfies equation \eqref{limiting_equation} if and only if  satisfies the following equation
	\begin{equation}\label{limiting_equation_alt}
		\le\{\begin{array}{l}
			\ds{\gamma\partial_{t}u(t)=\partial^2_x u(t) + \abs{\partial_x u(t)}_{H}^{2}u(t)+\frac{1}{2}\,\text{{\em tr}}_{K}\big(u(t)\times (u(t)\times \partial_{t}u(t))\big),\ \ \ \ t>0, }\\
			\vs 
			\ds{u(0,x)=u_{0}(x),\ \ \ \ u(t,0)=u(t,L)=0 }.
		\end{array}\r.
	\end{equation}
\end{lem}

\begin{proof}

If $u$ satisfies equation \eqref{limiting_equation_alt}, then we have 
\begin{equation*}
	\begin{array}{ll}
		&\ds{\gamma\partial_{t}u(t)=\partial^2_x u(t)+\abs{\partial_x u(t)}_{H}^{2}u+\frac{1}{2}\varphi\big(-\abs{u(t)}^{2}\partial_{t}u(t)+(u(t)\cdot \partial_{t}u(t))u(t)\big) }\\
		\vs
		&\ds{\quad\quad\quad\quad = \partial^2_x u(t)+\abs{\partial_x u(t)}_{H}^{2}u(t)+\frac{1}{2}\varphi\Big(-\partial_{t}(\abs{u(t)}^{2}u(t))+3(u(t)\cdot \partial_{t}u(t))u(t)\Big) },
	\end{array}
\end{equation*}
the identity holding $\mathbb{P}$-a.s. in $L^2(0,T;H)$.
Then, since 
\[\gamma u(t)\cdot \partial_{t}u(t)=\partial^2_x u(t) \cdot u(t)+ \abs{\partial_x u(t)}_{H}^{2}\abs{u(t)}^{2},\] we have 
\begin{equation*}
	\gamma (u(t)\cdot \partial_{t}u(t))u(t)=(\partial^2_x u(t)\cdot u(t))u(t)+\abs{\partial_x u(t)}_{H}^{2}\abs{u(t)}^{2}u(t).
\end{equation*}
This implies that 
\begin{equation*}
	\begin{array}{l}
	\ds{\gamma\partial_{t}u(t)+\frac{1}{2}\varphi\partial_{t}(\abs{u(t)}^{2}u(t))}\\
	\vs 
	\ds{\quad \quad \quad =\partial^2_x u(t)+ \abs{\partial_x u(t)}_{H}^{2}u(t)+\frac{3\varphi}{2\gamma}(\partial^2_x u(t)\cdot u(t))u(t)+\frac{3\varphi}{2\gamma}\abs{\partial_x u(t)}_{H}^{2}\abs{u(t)}^{2}u(t).}	
	\end{array}
\end{equation*}

On the other hand, if $u$ is a solution of equation \eqref{limiting_equation}, in order to prove that it is also a solution to \eqref{limiting_equation_alt}  it suffices to show that 
\begin{equation*}
	\gamma u(t)\cdot \partial_{t}u(t)=\partial^2_x u(t)\cdot u(t)+\abs{\partial_x u(t)}_{H}^{2}\abs{u(t)}^{2}.
\end{equation*}
Indeed, note that 
\begin{equation*}
	\begin{array}{l}
	\ds{\gamma\partial_{t}u(t)+\frac{1}{2}\varphi\Big(2(u(t)\cdot \partial_{t}u(t))u(t)+\abs{u(t)}^{2}\partial_{t}u(t)\Big)}\\
	\vs 
	\ds{\quad \quad \quad \quad \quad \quad =\partial^2_x u(t)+ \abs{\partial_x u(t)}_{H}^{2}u(t)+\frac{3\varphi}{2\gamma}(\partial^2_x u(t)\cdot u(t))u(t)+\frac{3\varphi}{2\gamma}\abs{\partial_x u(t)}_{H}^{2}\abs{u(t)}^{2}u(t),}	
	\end{array}
\end{equation*}
so that if we take the scalar product by $u$ of both sides, we get 
\begin{equation*}
	\Big(1+\frac{3}{2\gamma}\varphi\abs{u(t)}^{2}\Big)\Big((\gamma u(t)\cdot \partial_{t}u(t))-(\partial^2_x u(t)\cdot u(t))-\abs{\partial_x u(t)}_{H}^{2}\abs{u(t)}^{2}\Big)=0,
\end{equation*}
which completes the proof.

\end{proof}

\begin{rem}
	By using the \eqref{limiting_equation_alt} formulation of equation \eqref{limiting_equation}, it is immediate to check that if $u_0\in  M$ then $u(t)\in M$, for every $t \in\,[0,T]$. Actually, for any $t>0$ we have 
	\begin{equation*}
		\begin{array}{ll}
			&\ds{ \frac{1}{2}\frac{d}{dt}\Big(\abs{u(t)}_{H}^{2}-1\Big)=\Inner{\partial^2_x u(t),u(t)}_{H}+\abs{\partial_x u(t)}_{H}^{2}\abs{u(t)}_{H}^{2}+\frac{1}{2}\Inner{\varphi(u(t)\times (u(t)\times \partial_{t}u(t))),u(t)}_{H}  }\\
			\vs
			&\ds{\quad\quad\quad\quad\quad\quad\quad \quad\quad\quad\quad =\abs{\partial_x u(t)}_{H}^{2}\Big(\abs{u(t)}_{H}^{2}-1\Big) }.
		\end{array}
	\end{equation*}
	Combined with the fact that  $\abs{u_0}_{H}-1=0$ and $\partial_x u \in\,C([0,T];H)$, this implies that $\abs{u(t)}_{H}=1$, for any $t\geq0$.
\end{rem}

\begin{lem}
	Let $u$ be a solution to equation \eqref{limiting_equation}, with  $u_0\in H^{1}\cap M$. Then for every $t\geq0$ we have
	\begin{equation}\label{lim_est}
		\abs{u(t)}_{H^{1}}^{2}+2\gamma\int_{0}^{t}\abs{\partial_{t}u(s)}_{H}^{2}\ ds\leq \abs{u_0}_{H^{1}}^{2}.
	\end{equation}
\end{lem}

\begin{proof} If we use the \eqref{limiting_equation_alt} formulation of equation \eqref{limiting_equation}, we get
\begin{equation*}
	\begin{array}{l}
	\ds{\gamma \abs{\partial_{t}u(t)}_{H}^{2}=\Inner{\partial^2_x u(t), \partial_{t}u(t)}_{H}+\Inner{\abs{\partial_x u(t)}_{H}^{2}u(t),\partial_{t}u(t) }_{H}}\\
	\vs 
	\ds{\quad \quad \quad \quad \quad \quad \quad \quad-\frac{1}{2}\Inner{\varphi\abs{u(t)}^{2}\partial_{t}u(t),\partial_{t}u(t)}_{H}+\frac{1}{2}\Inner{\varphi(u(t)\cdot\partial_{t}u(t))u(t),\partial_{t}u(t) }_{H}.} 	
	\end{array}
\end{equation*}
	Recalling that $\vert u(t)\vert_H=1$, this gives 
		\[
		\begin{array}{l}
		\ds{	\gamma\,\abs{\partial_{t}u(t)}_{H}^{2} = -\frac{1}{2}\frac{d}{dt}\abs{u(t)}_{H^{1}}^{2}-\frac{1}{2}\Inner{\varphi\abs{u(t)}^{2}\partial_{t}u(t),\partial_{t}u(t)}_{H}+\frac{1}{2}\Inner{\varphi(u(t)\cdot\partial_{t}u(t))u(t),\partial_{t}u(t) }_{H}}\\
		\vs 
		\ds{\quad \quad \quad\quad \quad \quad\quad \quad \quad\quad \quad \quad\quad \quad \quad \leq -\frac{1}{2}\frac{d}{dt}\abs{u(t)}_{H^{1}}^{2},}
		\end{array}
\]
		and \eqref{lim_est} follows once we integrate both sides in time.

\end{proof}

\begin{Proposition}
	Let $u_{1}$ and $u_{2}$ be any two solutions of equation \eqref{limiting_equation}, with initial conditions $u_{1,0}, u_{2,0}\in H^{2}\cap M$, respectively. Then there exist some constants $c_{1},c_{2}>0$, depending only on $u_{1,0}$, $u_{2,0}$ and $\varphi$, such that for every $t\geq 0$
	\begin{equation}\label{comparison}
		\abs{u_{1}(t)-u_{2}(t)}_{H^{1}}^{2}+\int_{0}^{t}\abs{\partial_{t}u_{1}(s)-\partial_{t}u_{2}(s)}_{H}^{2}ds\leq c_{1}\abs{u_{1,0}-u_{2,0}}_{H^{1}}^{2}e^{c_{2}t}.
	\end{equation}
	In particular,  there is at most one solution to equation \eqref{limiting_equation} in $C([0,T];H^1)\cap L^2(0,T;H^2)$, with $\partial_t u \in\,L^2(0,T;H)$.
\end{Proposition}

\begin{proof}

Let us fix $u_{1,0},u_{2,0}\in H^{1}\cap M$ and let $u_{1},u_{2}$ be solutions of equation \eqref{limiting_equation} with initial conditions $u_{1,0},u_{2,0}$, respectively. If we denote $v_{1}=\partial_{t}u_{1}$ and $v_{2}=\partial_{t}u_{2}$, then, recalling \eqref{sm10} 
\begin{equation*}
	\begin{array}{l}
		\ds{\gamma \abs{(v_{1}-v_{2})(t)}_{H}^{2}= \Inner{\partial^2_x (u_{1}-u_{2})(t), (v_{1}-v_{2})(t) }_{H}}\\
		\vs 
		\ds{\quad \quad  \quad \quad +\Inner{\abs{\partial_x u_{1}(t)}_{H}^{2}u_{1}(t)-\abs{\partial_x u_{2}(t)}_{H}^{2}u_{2}(t),(v_{1}-v_{2})(t) }_{H} }\\
		\vs
		\ds{\quad \quad  \quad \quad \quad \quad  - \frac{1}{2}\Inner{\varphi\big(\abs{u_{1}(t)}^{2}v_{1}(t)-\abs{u_{2}(t)}^{2}v_{2}(t)\big),(v_{1}-v_{2})(t) }_{H}}\\
		\vs 
		\ds{\quad \quad \quad \quad \quad \quad \quad \quad +\frac{1}{2}\Inner{\varphi\big( (u_{1}(t)\cdot v_{1}(t))u_{1}(t)-(u_{2}(t)\cdot v_{2}(t))u_{2}(t) \big),(v_{1}-v_{2})(t) }_{H}.}
	\end{array}
\end{equation*}
Hence, if $a$ is an arbitrary positive constant and we denote
\begin{equation}
	Y_{a}(t):=\exp\left(-a\int_{0}^{t}\big(1+\abs{v_{2}(s)}_{H}^{2}\big)ds\right),\ \ \ \ t\geq 0,
\end{equation}
we get
\begin{equation}\label{sm117}
	\begin{array}{l}
		\ds{ \frac{d}{dt}\Big(Y_{a}(t)\abs{(u_{1}-u_{2})(t)}_{H^{1}}^{2}\Big)= Y_{a}(t)\Big(\frac{d}{dt}\abs{(u_{1}-u_{2})(t)}_{H^{1}}^{2}-a\big(1+\abs{v_{2}(t)}_{H}^{2}\big)\abs{(u_{1}-u_{2})(t)}_{H^{1}}^{2}\Big)  }\\
		\vs
		\ds{\quad=Y_{a}(t) \Big( -a\big(1+\abs{v_{2}(t)}_{H}^{2}\big)\abs{(u_{1}-u_{2})(t)}_{H^{1}}^{2}+2\Inner{\abs{u_{1}(t)}_{H^{1}}^{2}u_{1}(t)-\abs{u_{2}(t)}_{H^{1}}^{2}u_{2}(t),(v_{1}-v_{2})(t) }_{H}   }\\
		\vs
		\ds{\quad \quad \quad \quad -2\gamma \abs{(v_{1}-v_{2})(t)}_{H}^{2}-\Inner{\varphi\big(\abs{u_{1}(t)}^{2}v_{1}(t)-\abs{u_{2}(t)}^{2}v_{2}(t)\big),(v_{1}-v_{2})(t) }_{H} }\\
		\vs
		\ds{\quad \quad \quad \quad \quad \quad \quad \quad +\Inner{\varphi\big( (u_{1}(t)\cdot v_{1}(t))u_{1}(t)-(u_{2}(t)\cdot v_{2}(t))u_{2}(t) \big),(v_{1}-v_{2})(t) }_{H} \Big) }.
	\end{array}
\end{equation}
Now, for any $\epsilon>0$ we can find $c(\epsilon)>0$ such that
\begin{equation}\label{sm118}
	\begin{array}{l}
		\ds{\Inner{\abs{u_{1}(t)}_{H^{1}}^{2}u_{1}(t)-\abs{u_{2}(t)}_{H^{1}}^{2}u_{2}(t),(v_{1}-v_{2})(t) }_{H}    }\\
		\vs
		\ds{\quad\quad =\abs{u_{1}(t)}_{H^{1}}^{2}\Inner{(u_{1}-u_{2})(t),(v_{1}-v_{2})(t)}_{H}+\Inner{\big(\abs{u_{1}(t)}_{H^{1}}^{2}-\abs{u_{2}(t)}_{H^{1}}^{2}\big)u_{2}(t),(v_{1}-v_{2})(t) }_{H} }\\
		\vs
		\ds{\quad\quad \leq \abs{u_{1}(t)}_{H^{1}}^{2}\Big(\epsilon\abs{(v_{1}-v_{2})(t)}_{H}^{2}+c(\epsilon)\abs{(u_{1}-u_{2})(t)}_{H}^{2}\Big)}\\
		\vs
		\ds{\quad \quad \quad \quad +c\abs{u_{2}(t)}_{H}\big(\abs{u_{1}(t)}_{H^{1}}+\abs{u_{2}(t)}_{H^{1}}\big)\Big(\epsilon\abs{(v_{1}-v_{2})(t)}_{H}^{2}+c(\epsilon)\abs{(u_{1}-u_{2})(t)}_{H^{1}}^{2}\Big) }.
	\end{array}
\end{equation}
Moreover
\begin{equation}\label{sm119}
	\begin{array}{l}
		\ds{-\Inner{\varphi\big(\abs{u_{1}(t)}^{2}v_{1}(t)-\abs{u_{2}(t)}^{2}v_{2}(t)\big),(v_{1}-v_{2})(t) }_{H}  }\\
		\vs
		\ds{\quad\quad =-\Inner{\varphi\abs{u_{1}(t)}^{2}(v_{1}-v_{2})(t),(v_{1}-v_{2})(t) }_{H}-\Inner{\varphi\big(\abs{u_{1}(t)}^{2}-\abs{u_{2}(t)}^{2}\big)v_{2}(t),(v_{1}-v_{2})(t) }_{H}  }\\
		\vs
		\ds{\quad\quad \leq -\Inner{\varphi\abs{u_{1}(t)}^{2}(v_{1}-v_{2})(t),(v_{1}-v_{2})(t) }_{H}}\\
		\vs
		\ds{\quad \quad \quad \quad \quad +\vert \varphi\vert_{\infty}\big(\abs{u_{1}(t)}_{H^{1}}+\abs{u_{2}(t)}_{H^{1}}\big)\abs{(u_{1}-u_{2})(t)}_{H^{1}}\abs{v_{2}(t)}_{H}\abs{(v_{1}-v_{2})(t)}_{H} }\\
		\vs
		\ds{\quad\quad \leq -\Inner{\varphi\abs{u_{1}(t)}^{2}(v_{1}-v_{2})(t),(v_{1}-v_{2})(t) }_{H}}\\
		\vs
		\ds{\quad\quad\quad\quad\quad+c\vert \varphi\vert_{\infty}\big(\abs{u_{1}(t)}_{H^{1}}+\abs{u_{2}(t)}_{H^{1}}\big)\Big(\epsilon\abs{(v_{1}-v_{2})(t)}_{H}^{2}+c(\epsilon)\abs{v_{2}(t)}_{H}^{2}\abs{(u_{1}-u_{2})(t)}_{H^{1}}^{2}\Big)  },
	\end{array}
\end{equation}
and
\begin{equation}\label{sm120}
	\begin{array}{l}
		\ds{\Inner{\varphi\big( (u_{1}(t)\cdot v_{1}(t))u_{1}(t)-(u_{2}(t)\cdot v_{2}(t))u_{2}(t) \big),(v_{1}-v_{2})(t) }_{H}  }\\
		\vs
		\ds{\quad\quad = \Inner{\varphi\big(u_{1}(t)\cdot (v_{1}-v_{2})(t)\big)u_{1}(t),(v_{1}-v_{2})(t) }_{H}}\\
		\vs
		\ds{\quad +\Inner{\varphi\big((u_{1}-u_{2})(t)\cdot v_{2}(t)\big)u_{1}(t),(v_{1}-v_{2} )(t)}_{H} + \Inner{\varphi(u_{2}(t)\cdot v_{2}(t))(u_{1}-u_{2})(t),(v_{1}-v_{2})(t) }_{H} }\\
		\vs
		\ds{\quad\quad \leq \Inner{\varphi\big(u_{1}(t)\cdot (v_{1}-v_{2})(t)\big)u_{1}(t),(v_{1}-v_{2})(t) }_{H}}\\
		\vs 
		\ds{\quad\quad \quad\quad  +c\vert\varphi\vert_{\infty}\big(\abs{u_{1}(t)}_{H^{1}}+\abs{u_{2}(t)}_{H^{1}}\big)\Big(\epsilon\abs{(v_{1}-v_{2})(t)}_{H}^{2}+c(\epsilon)\abs{v_{2}(t)}_{H}^{2}\abs{(u_{1}-u_{2})(t)}_{H^{1}}^{2}\Big) }.
	\end{array}
\end{equation}
Since $\varphi\geq 0$, we have 
\begin{equation*}
	-\Inner{\varphi\abs{u_{1}}^{2}(v_{1}-v_{2}),v_{1}-v_{2}}_{H}+\Inner{\varphi\big(u_{1}\cdot (v_{1}-v_{2})\big)u_{1},v_{1}-v_{2} }_{H}\leq 0.
\end{equation*}
Hence, thanks to \eqref{lim_est} we can take $\bar{\epsilon}>0$ sufficiently small and $\bar{a}=\bar{a}(\bar{\epsilon})>0$ sufficiently large  so that if we replace \eqref{sm118}, \eqref{sm119} and \eqref{sm120} into \eqref{sm117}, we get
\begin{equation}
	\begin{array}{ll}
		&\ds{ Y_{\bar{a}}(t)\abs{u_{1}(t)-u_{2}(t)}_{H^{1}}^{2}+c\int_{0}^{t}Y_{\bar{a}}(s)\big(1+\abs{v_{2}(s)}_{H}^{2}\big)\abs{u_{1}(s)-u_{2}(s)}_{H^{1}}^{2}ds  }\\
		\vs
		&\ds{\quad\quad\quad\quad\quad\quad\quad\quad\quad  +c\int_{0}^{t}Y_{\bar{a}}(s)\abs{v_{1}(s)-v_{2}(s)}_{H}^{2}ds \leq \abs{u_{1,0}-u_{2,0}}_{H^{1}}^{2}  },
	\end{array}
\end{equation}
for some constant $c=c(u_{1,0},u_{2,0},\varphi)>0$. Finally, since from \eqref{lim_est}
\begin{equation*}
	Y_{\bar{a}}(t)\geq \exp\Big(-\bar{a}\Big(t+\abs{u_{2,0}}_{H^{1}}^{2}/2\gamma\Big)\Big), \ \ \ \ t> 0,
\end{equation*}
we can complete the proof of \eqref{comparison}.

\end{proof}

\section{Proof of the validity
 of the small-mass limit}
 \label{sec7}
 
In this final section, we conclude the proof of Theorem \ref{limite}. We first prove some identities, then we investigate tightness and finally we proceed with the proof of the theorem.

\subsection{An identity for the solution of system  \eqref{SPDE1}}

\begin{Lemma} For every $\mu >0$ and $(u_0,v_0) \in\,\mathcal{H}_1\cap \mathcal{M}$ the solution $(u_\mu(t),v_\mu(t))$ of system \eqref{SPDE1} (or, equivalently system \eqref{SPDE1-bis}) satisfies the following identity, for every $t\geq 0$
\begin{equation}\label{sm1200}
	\begin{array}{l}
		\ds{ \gamma u_{\mu}(t)+\frac{1}{2}\varphi\abs{u_{\mu}(t)}^{2}u_{\mu}(t)+\mu v_{\mu}(t)}\\
		\vs 
		\ds{\quad  = \gamma u_0+\frac{1}{2}\varphi\abs{u_0}^{2}u_0+\mu v_0+\int_{0}^{t}\partial^2_x u_{\mu}(s)ds+\int_{0}^{t}\abs{\partial_x u_{\mu}(s)}_{H}^{2}u_{\mu}(s)ds  }\\
		\vs
		\ds{\quad\quad +\frac{3}{2\gamma}\varphi\int_{0}^{t}\big(\partial^2_x u_{\mu}(s)\cdot u_{\mu}(s)\big)u_{\mu}(s)ds+\frac{3}{2\gamma}\varphi\int_{0}^{t}\abs{\partial_x u_{\mu}(s)}_{H}^{2}\abs{u_{\mu}(s)}^{2}u_{\mu}(s)ds+R_{\mu}(t)   },
	\end{array}
\end{equation}
where 
\begin{equation}\label{sm123}
	\begin{array}{l}
		\ds{ R_{\mu}(t) = \frac{3\mu}{2\gamma}\varphi(u_0\cdot v_0)u_0-\frac{3\mu}{2\gamma}\varphi(u_{\mu}(t)\cdot v_{\mu}(t))u_{\mu}(t)-\mu\int_{0}^{t}\abs{v_{\mu}(s)}_{H}^{2}u_{\mu}(s)ds  }\\
		\vs
		\ds{\quad\quad\quad\quad +\frac{3\mu}{2\gamma}\varphi\int_{0}^{t}(u_{\mu}(s)\cdot v_{\mu}(t))v_{\mu}(s)ds + \frac{3\mu}{2\gamma}\varphi\int_{0}^{t}\abs{v_{\mu}(s)}^{2}u_{\mu}(s)ds  }\\
		\vs
		\ds{\quad\quad\quad\quad \quad \quad- \frac{3\mu}{2\gamma}\varphi\int_{0}^{t}\abs{v_{\mu}(s)}_{H}^{2}\abs{u_{\mu}(s)}^{2}u_{\mu}(s)ds+\sqrt{\mu}\int_{0}^{t}(u_{\mu}(s)\times v_{\mu}(s))dw(s)       }\\
		\vs
		\ds{\quad\quad \quad\quad\quad\quad \quad \quad\quad\quad\quad\quad =: \frac{3\mu}{2\gamma}\varphi(u_0\cdot v_0)u_0+\sum_{i=1}^{6}J_{\mu,i}(t) }.
	\end{array}
\end{equation}	
\end{Lemma}

\begin{proof}  In view of  \eqref{SPDE1-bis}, we have 
\begin{equation*}
	\begin{array}{l}
		\ds{d(u_{\mu}(t)\cdot v_{\mu}(t))=\abs{v_{\mu}(t)}^{2}dt+\frac{1}{\mu}\big(\partial^2_x u_{\mu}(t)\cdot u_{\mu}(t)\big)dt+\frac{1}{\mu}\abs{\partial_x u_{\mu}(t)}_{H}^{2}\abs{u_{\mu}(t)}^{2}dt}\\
		\vs
		\ds{-\abs{v_{\mu}(t)}_{H}^{2}\abs{u_{\mu}(t)}^{2}dt -\frac{\gamma}{\mu}(u_{\mu}(t)\cdot v_{\mu}(t))dt   +\frac{1}{2\mu}\varphi\Big(-\abs{u_{\mu}(t)}^{2}v_{\mu}(t)+(u_{\mu}(t)\cdot v_{\mu}(t))u_{\mu}(t)\Big)\cdot u_{\mu}(t)dt }\\
		\vs 
		\ds{\quad \quad \quad \quad \quad \quad+\frac{1}{\sqrt{\mu}}u_{\mu}(t)\cdot (u_{\mu}(t)\times v_{\mu}(t))dw(t) =\abs{v_{\mu}(t)}^{2}dt+\frac{1}{\mu}\big(\partial^2_x u_{\mu}(t)\cdot u_{\mu}(t)\big)dt}\\
		\vs
		\ds{\quad\quad \quad\quad \quad\quad \quad\quad \quad\quad +\frac{1}{\mu}\abs{\partial_x u_{\mu}(t)}_{H}^{2}\abs{u_{\mu}(t)}^{2}dt-\abs{v_{\mu}(t)}_{H}^{2}\abs{u_{\mu}(t)}^{2}dt-\frac{\gamma}{\mu}(u_{\mu}(t)\cdot v_{\mu}(t))dt  }.
	\end{array}
\end{equation*}
This implies that 
\begin{equation*}
	\begin{array}{l}
		\ds{d\Big((u_{\mu}(t)\cdot v_{\mu}(t))u_{\mu}(t)\Big)}\\
		\vs 
		\ds{\quad =(u_{\mu}(t)\cdot v_{\mu}(t))v_{\mu}(t)dt+\abs{v_{\mu}(t)}^{2}u_{\mu}(t)dt+\frac{1}{\mu}\big(\partial^2_x u_{\mu}(t)\cdot u_{\mu}(t)\big)u_{\mu}(t)dt   }\\
		\vs
		\ds{\quad\quad\quad\quad +\frac{1}{\mu}\abs{\partial_x u_{\mu}(t)}_{H}^{2}\abs{u_{\mu}(t)}^{2}u_{\mu}(t)dt-\abs{v_{\mu}(t)}_{H}^{2}\abs{u_{\mu}(t)}^{2}u_{\mu}(t)dt-\frac{\gamma}{\mu}(u_{\mu}(t)\cdot v_{\mu}(t))u_{\mu}(t)dt }.
	\end{array}
\end{equation*}
Hence, for every $t\geq 0$, we obtain
\begin{equation}
	\begin{array}{l}
		\ds{\gamma \int_{0}^{t}(u_{\mu}(s)\cdot v_{\mu}(s))u_{\mu}(s)ds = \int_{0}^{t}\big(\partial^2_x u_{\mu}(s)\cdot u_{\mu}(s)\big)u_{\mu}(s)ds +\int_{0}^{t}\abs{\partial_x u_{\mu}(s)}_{H}^{2}\abs{u_{\mu}(s)}^{2}u_{\mu}(s)ds  }\\
		\vs
		\ds{\quad-\mu(u_{\mu}(t)\cdot v_{\mu}(t))u_{\mu}(t)+\mu(u_0\cdot v_0)u_0+\mu \int_{0}^{t}(u_{\mu}(s)\cdot v_{\mu}(s))v_{\mu}(s)ds+\mu\int_{0}^{t}\abs{v_{\mu}(s)}^{2}u_{\mu}(s)ds   }\\
		\vs
		\ds{\quad\quad \quad\quad \quad\quad \quad\quad -\mu\int_{0}^{t}\abs{v_{\mu}(s)}_{H}^{2}\abs{u_{\mu}(s)}^{2}u_{\mu}(s)ds }.
	\end{array}
\end{equation}
By using the fact that 
\[d(\abs{u(t)}^{2}u(t))=2(u(t)\cdot v(t))u(t)+\abs{u(t)}^{2}v(t),\] this implies
\begin{equation}\label{sm122}
	\begin{array}{l}
		\ds{ \int_{0}^{t}\Big(-\abs{u_{\mu}(s)}^{2}v_{\mu}(s)+(u_{\mu}(s)\cdot v_{\mu}(s))u_{\mu}(s)\Big)ds }\\
		\vs
		\ds{= - \abs{u_{\mu}(t)}^{2}u_{\mu}(t)+\abs{u_0}^{2}u_0+3\int_{0}^{t}(u_{\mu}(s)\cdot v_{\mu}(s))u_{\mu}(s)ds  =-\abs{u_{\mu}(t)}^{2}u_{\mu}(t)+\abs{u_0}^{2}u_{0} }\\
		\vs
		\ds{\quad\quad +\frac{3}{\gamma}\int_{0}^{t}\big(\partial^2_x u_{\mu}(s)\cdot u_{\mu}(s)\big)u_{\mu}(s)ds +\frac{3}{\gamma}\int_{0}^{t}\abs{\partial_x u_{\mu}(s)}_{H}^{2}\abs{u_{\mu}(s)}^{2}u_{\mu}(s)ds   }\\
		\vs
		\ds{\quad\quad \quad -\frac{3\mu}{\gamma}(u_{\mu}(t)\cdot v_{\mu}(t))u_{\mu}(t)+\frac{3\mu}{\gamma}(u_0\cdot v_0)u_0+\frac{3\mu}{\gamma} \int_{0}^{t}(u_{\mu}(s)\cdot v_{\mu}(s))v_{\mu}(s)ds }\\
		\vs
		\ds{\quad\quad\quad\quad \quad\quad\quad+\frac{3\mu}{\gamma}\int_{0}^{t}\abs{v_{\mu}(s)}^{2}u_{\mu}(s)ds-\frac{3\mu}{\gamma}\int_{0}^{t}\abs{v_{\mu}(s)}_{H}^{2}\abs{u_{\mu}(s)}^{2}u_{\mu}(s)ds }.
	\end{array}
\end{equation}
Finally, we rewrite system \eqref{SPDE1-bis} in the following form
\begin{equation}\label{sm121}
	\begin{array}{l}
		\ds{ \gamma u_{\mu}(t)+\mu v_{\mu}(t)=\gamma u_0+\mu v_0 +\int_{0}^{t}\partial^2_x u_{\mu}(s)ds + \int_{0}^{t}\abs{\partial_x u_{\mu}(s)}_{H}^{2}u_{\mu}(s)ds-\mu \int_{0}^{t}\abs{v_{\mu}(s)}_{H}^{2}u_{\mu}(s)ds  }\\
		\vs
		\ds{\quad\quad\quad +\frac{1}{2}\,\varphi\int_{0}^{t}\Big(-\abs{u_{\mu}(s)}^{2}v_{\mu}(s)+(u_{\mu}(s)\cdot v_{\mu}(s))u_{\mu}(s)\Big)ds +\sqrt{\mu}\int_{0}^{t}(u_{\mu}(s)\times v_{\mu}(s))dw(s) }.
	\end{array}
\end{equation}
Therefore, if we replace \eqref{sm122} into \eqref{sm121}, we 
get \eqref{sm1200}, with $R_\mu(t)$ defined by \eqref{sm123}.
\end{proof}

\begin{lem}\label{reminder}
	For every $(u_0,v_0)\in \H_{1}\cap \M$ and $t>0$ we have 
	\begin{equation}
		\lim_{\mu\to0}\E\sup_{r\in[0,t]}\abs{R_{\mu}(r)}_{H}^{2}=0.
	\end{equation}
\end{lem}

\begin{proof}
First, note that, since
\begin{equation*}
	\abs{J_{\mu,1}(t)}_{H}\leq c\,\vert \varphi\vert_{\infty} \mu\abs{u_{\mu}(t)}_{H^{1}}^{2}\abs{v_{\mu}(t)}_{H},
\end{equation*}
due to \eqref{uniform_est1} we have 
\begin{equation}\label{error_est1}
	\lim_{\mu\to0} \E\sup_{r\in[0,t]}\abs{J_{\mu,1}(r)}_{H}^{2}=0,
\end{equation}
and since
\begin{equation*}
\abs{J_{\mu,2}(r)}_{H}^{2}\leq c\,t \mu^{2}\int_{0}^{t}\abs{v_{\mu}(s)}_{H}^{4}ds\leq c\,t \mu^{2}\int_{0}^{t}\abs{v_{\mu}(s)}_{H}^{4}\vert u_\mu(s)\vert_{H^1}^2\,ds,\ \ \ \ \ r \in\,[0,t],
\end{equation*}
due to  \eqref{uniform_est2} we have 
\begin{equation}\label{error_est2}
	\lim_{\mu\to0}\E\sup_{r\in[0,t]}\abs{J_{\mu,2}(r)}_{H}^{2}=0.
\end{equation}
Moreover, thanks to \eqref{uniform_est1}, \eqref{uniform_est2} and 
\begin{equation*}
\abs{J_{\mu,3}(r)}_{H}^{2}+\abs{J_{\mu,4}(r)}_{H}^{2}\leq c\,t\vert\varphi\vert_{\infty}^{2}\mu^{2}\int_{0}^{t}\abs{u_{\mu}(s)}_{H^{1}}^{2}\abs{v_{\mu}(s)}_{H^{1}}^{2}\abs{v_{\mu}(s)}_{H}^{2}ds,\ \ \ \ \ r \in\,[0,t],
\end{equation*}
we have
\begin{equation}\label{error_est3}
	\lim_{\mu\to0}\E\sup_{r\in[0,t]}\Big(\abs{J_{\mu,3}(r)}_{H}^{2}+\abs{J_{\mu,4}(r)}_{H}^{2}\Big)=0.
\end{equation}
From \eqref{uniform_est1}, \eqref{uniform_est2}, and 
\begin{equation*}
	\abs{J_{\mu,5}}_{H}^{2}\leq c\,t\vert\varphi\vert_{\infty}^{2} \mu^{2}\E\int_{0}^{t}\abs{u_{\mu}(s)}_{H^{1}}^{6}\abs{v_{\mu}(s)}_{H}^{4}ds,\ \ \ \ \ r \in\,[0,t],
\end{equation*}
it follows that 
\begin{equation}\label{error_est4}
\lim_{\mu\to 0}	\E\sup_{r\in[0,t]}\abs{J_{\mu,5}(r)}_{H}^{2}=0.
\end{equation}
Furthermore,  we have 
\begin{equation*}
	\begin{array}{ll}
		&\ds{ \E\sup_{r\in[0,t]}\abs{J_{\mu,6}(r)}_{H}^{2} \leq c\mu\E\int_{0}^{t}\norm{u_{\mu}(s)\times v_{\mu}(s)}_{\mathcal{T}_{2}(K,H)}^{2}ds\leq c\vert\varphi\vert_{\infty}\mu\E\int_{0}^{t}\abs{u_{\mu}(s)}_{H^{1}}^{2}\abs{v_{\mu}(s)}_{H}^{2}ds  },
	\end{array}
\end{equation*}
so that, thanks again to \eqref{uniform_est1}, 
\begin{equation}\label{error_est5}
	\lim_{\mu\to 0}	\E\sup_{r\in[0,t]}\abs{J_{\mu,6}(r)}_{H}^{2}=0.
\end{equation}
Finally, combining \eqref{error_est1}-\eqref{error_est5}, we complete our proof.
\end{proof}

\subsection{Tightness}

\begin{prop}\label{tightness}
	For every $(u_0,v_0)\in \H_{1}\cap\M$ and $T>0$, the family of probability measures $\big\{ \L(u_{\mu})\big\}_{\mu\in(0,1)}$ is tight in $C([0,T];H^{\delta})$, for every $\delta<2$.
\end{prop}

\begin{proof}
	According to  \eqref{uniform_est4-bis}, we have that 
	\begin{equation}\label{uniform_est}
		\sup_{\mu\in(0,1)} \E\left(\sup_{t\in[0,T]}\abs{u_{\mu}(t)}_{H^{2}}^{2}+\int_{0}^{T}\abs{\partial_{t}u_{\mu}(s)}_{H^{1}}^{2}ds\right)<+\infty.
	\end{equation}
	This, in particular, implies that for every $\epsilon>0$ there exists $L_{\epsilon}>0$ such that if we denote by $K_{\epsilon}$ the ball of radius $L_{\epsilon}$ in $C([0,T];H^{2})\cap W^{1,2}(0,T;H^{1})$, then 
	\begin{equation*}
		\inf_{\mu\in(0,1)}\P(u_{\mu}\in K_{\epsilon})\geq 1-\epsilon.
	\end{equation*}
Due to the Aubin-Lions lemma, we know that the set $K_{\epsilon}$ is compact in $C([0,T];H^{\delta})$, for every $\delta<2$.
\end{proof}

\subsection{Proof of Theorem \ref{limite}}

\begin{proof}
Thanks to Proposition \ref{tightness} and \eqref{uniform_est4-bis}, we have the family $\big\{ \L(u_{\mu},\mu\,\partial_{t}u_{\mu})\big\}_{\mu\in(0,1)}$ is tight in $C([0,T];H^{\delta})\times C([0,T];H^{1})$, for any $\delta<2$. If for every $T>0$ and $\delta<2$ we define 
\begin{equation*}
	\Gamma_{T,\delta}:=C([0,T];H^{\delta})\times C([0,T];H^{1})\times C([0,T];E),
\end{equation*}
where $E$ is any Banach space such that the embedding $K\subset E$ is Hilbert-Schmidt, then as a consequence of  Skorokhod's theorem, for any sequence $\{\mu_{k}\}_{k\in\mathbb{N}}$, converging to zero, there exists a subsequence, still denoted by $\{\mu_{k}\}_{k\in\mathbb{N}}$, and some $\Gamma_{T,\delta}$-valued random variables 
\begin{equation*}
	\mathcal{Y}_{k}:=\big(\varrho_{k}, \mu_k\vartheta_{k}, \hat{w}_{k} \big),\ \ \ \ \mathcal{Y}:=\big(\varrho, \vartheta, \hat{w}\big),\ \ \ \ k\in\mathbb{N},
\end{equation*}
all defined on some probability space $\big(\hat{\Omega},\hat{\F},\{\hat{\F}_{t}\}_{t\in[0,T]},\hat{\P} \big)$, such that
\begin{equation}\label{law_coincide}
	\L(\mathcal{Y}_{k})=\L\big(u_{\mu_{k}},\mu_{k}\partial_{t}u_{\mu_{k}},w \big),\ \ \ \ k\in\mathbb{N},
\end{equation}
and 
\begin{equation}\label{convergence_as}
	\lim_{k\to\infty}\Big( \abs{\varrho_{k}-\varrho}_{C([0,T];H^{\delta})}+\mu_k\abs{\vartheta_{k}}_{C([0,T];H^{1})} + \abs{\hat{w}_{k}-\hat{w}}_{C([0,T];E)} \Big)=0,\ \ \ \ \ \hat{\mathbb{P}}-\text{a.s.}.
\end{equation}
In particular, due to \eqref{uniform_est1}, \eqref{uniform_est4-bis}, \eqref{uniform_est} and \eqref{law_coincide}, we have 
\begin{equation}\label{bound1}
	\sup_{k\in\mathbb{N}} \Bigg(\hat{\E}\sup_{t\in[0,T]}\Big(\abs{\varrho_{k}(t)}_{H^{2}}^{2}+\mu_{k}\abs{\vartheta_{k}(t)}_{H^{1}}^{2}\Big)+\hat{\E}\int_{0}^{T}\abs{\vartheta_{k}(s)}_{H^{1}}^{2}ds\Bigg)<+\infty,
\end{equation}
and there exists a deterministic $c>0$ such that
\begin{equation}\label{bound2}
	\sup_{k\in\mathbb{N}}\Big( \sup_{t\in[0,T]}\abs{\varrho_{k}(t)}_{H^{1}}+\int_0^T \vert \partial_t\varrho_k(s)\vert_H^2\,ds\Big)\leq c,\ \ \ \ \hat{\P}\text{-a.s.}
\end{equation}
Thanks to \eqref{convergence_as} and \eqref{bound2}, it follows that $\varrho\in L^{2}(\hat{\Omega};L^{\infty}(0,T;H^{\delta}))$, for every $\delta<2$, and 
\begin{equation}\label{bound3}
	\sup_{t\in[0,T]}\abs{\varrho(t)}_{H^{1}}\leq c, \ \ \ \ \hat{\P}\text{-a.s.}.
\end{equation}
Moreover, from \eqref{bound1} and \eqref{bound2} we get that $\varrho$ is weakly differentiable in time, with $\partial_t\varrho \in\,L^2(0,T;H)$ and
\begin{equation}
\int_0^T \vert \partial_t \varrho(s)\vert_H^2\,ds\leq c,\ \ \ \ \ \hat{\P}\text{-a.s.}\end{equation}
Finally, as a consequence of \eqref{convergence_as} and \eqref{bound1}, we have
\begin{equation}
\label{bound7}
\int_0^T \hat{\mathbb{E}	}\vert \varrho(s)\vert_{H^2}^2\,ds<\infty.
\end{equation}

Now, if we can show that $\varrho$ solves equation \eqref{limiting_equation}, then by the  uniqueness of solutions for equation \eqref{limiting_equation}, due to the classical argument by Gyongy and Krylov (see \cite{gk}), we can conclude that $u_{\mu_k}$ converges to $u$ in $C([0,T];H^{\delta})$, for every sequence $\mu_k\downarrow 0$, and the convergence is in probability.

\medskip

Due to \eqref{law_coincide} and identities \eqref{sm1200} and \eqref{sm123}, we have that for every $\psi \in\,C^\infty_0([0,L])$
\begin{equation*}
	\begin{array}{l}
		\ds{ \langle \gamma\varrho_{k}(t)+\frac{1}{2}\varphi\abs{\varrho_{k}(t)}^{2}\varrho_{k}(t)+\mu_k\vartheta_{k}(t),\psi\rangle_H=\langle\gamma u_{0}+\frac{1}{2}\varphi\abs{u_{0}}^{2}u_{0}+\mu_k v_{0},\psi\rangle_H+\int_{0}^{t}\langle\partial^2_x \varrho_{k}(s),\psi\rangle_Hds }\\
		\vs
		\ds{\quad \quad +\int_{0}^{t}\langle\abs{\varrho_{k}(s)}_{H^{1}}^{2}\varrho_{k}(s),\psi\rangle_Hds +\frac{3}{2\gamma}\int_{0}^{t}\langle \varphi\big(\partial^2_x\varrho_{k}(s)\cdot \varrho_{k}(s)\big)\varrho_{k}(s),\psi\rangle_Hds}\\
		\vs 
		\ds{ \quad \quad \quad \quad \quad \quad \quad+\frac{3}{2\gamma}\int_{0}^{t}\langle \varphi\abs{\varrho_{k}(s)}_{H^{1}}^{2}\abs{\varrho_{k}(s)}^{2}\varrho_{k}(s),\psi\rangle_Hds+\langle \hat{R}_{k}(t),\psi\rangle_H  },
	\end{array}
\end{equation*}
where 
\begin{equation*}
	\begin{array}{ll}
		&\ds{ \widehat{R}_{k}(t) = \frac{3\mu_k}{2\gamma}\varphi(u_0\cdot v_0)u_0-\frac{3\varphi\mu_k}{2\gamma}(\varrho_{k}(t)\cdot \vartheta_{k}(t))\varrho_{k}(t)-\mu_{k}\int_{0}^{t}\abs{\vartheta_{k}(s)}_{H}^{2}\varrho_{k}(s)ds  }\\
		\vs
		&\ds{\quad\quad\quad\quad +\frac{3\varphi \mu_k}{2\gamma}\int_{0}^{t}(\varrho_{k}(s)\cdot \vartheta_{k}(t))\vartheta_{k}(s)ds + \frac{3\varphi\mu_k}{2\gamma}\int_{0}^{t}\abs{\vartheta_{k}(s)}^{2}\varrho_{k}(s)ds  }\\
		\vs
		&\ds{\quad\quad\quad\quad - \frac{3\mu_k}{2\gamma}\int_{0}^{t}\abs{\vartheta_{k}(s)}_{H}^{2}\abs{\varrho_{k}(s)}^{2}\varrho_{k}(s)ds+\sqrt{\mu_{k}}\int_{0}^{t}(\varrho_{k}(s)\times \vartheta_{k}(s))d\hat{w}_{k}(s)       }.
	\end{array}
\end{equation*}
We have
\begin{equation*}
	\begin{array}{ll}
		&\ds{ \sup_{t\in[0,T]}\Big\lvert \varphi\Big(\abs{\varrho_{k}(t)}^{2}\varrho_{k}(t)-\abs{\varrho(t)}^{2}\varrho(t)\Big) \Big\rvert_{H}  }\\
		\vs
		&\ds{\quad \quad \leq c\vert\varphi\vert_{\infty}\sup_{t\in[0,T]}\Big(\big\lvert \big(\abs{\varrho_{k}(t)}^{2}-\abs{\varrho(t)}^{2}\big)\varrho_{k}(t)\big\rvert_{H} +\big\lvert \abs{\rho(t)}^{2}(\varrho_{k}(t)-\varrho(t))\big\rvert_{H}\Big)  }\\
		\vs
		&\ds{\quad \quad \leq c\vert\varphi\vert_{\infty}\sup_{t\in[0,T]}\Big(\big(\abs{\varrho_{k}(t)}_{H^{1}}^{2}+\abs{\varrho(t)}_{H^{1}}^{2}\big)\abs{\varrho_{k}(t)-\varrho(t)}_{H} \Big). }
	\end{array}
\end{equation*}
Then, in view of \eqref{convergence_as}, \eqref{bound2} and \eqref{bound3} we conclude
\begin{equation}\label{key1}
	\lim_{k\to\infty}\sup_{t \in\,[0,T]}\left|\langle \gamma\varrho_{k}(t)+\frac{1}{2}\varphi\abs{\varrho_{k}(t)}^{2}\varrho_{k}(t)+\mu_k \vartheta_{k}(t),\psi\rangle_H-\langle \gamma \varrho(t)+\frac{1}{2}\varphi\abs{\varrho(t)}^{2}\varrho(t),\psi\rangle_H\right|=0,\ \ \ \ \ \hat{\P}\text{-a.s.}
\end{equation}
Since
\begin{equation*}
	\sup_{t\in[0,T]}\Big\lvert \int_{0}^{t}\langle \partial^2_x(\varrho_{k}-\varrho)(s),\psi\rangle_H ds \Big\rvert\leq \vert \psi\vert_{H^1}\int_{0}^{T}\abs{\varrho_{k}(s)-\varrho(s)}_{H^{1}}ds,
\end{equation*}
due to \eqref{convergence_as} we have 
\begin{equation}\label{key2}
	\lim_{k\to\infty} \sup_{t \in\,[0,T]}\left|\int_{0}^{t}\langle \partial^2_x\varrho_{k}(s),\psi\rangle_H ds-\int_{0}^{t}\langle \partial^2_x\varrho(s),\psi\rangle_H\,ds\right|=0,\ \ \ \ \ \ \ \ \widehat{\P}\text{-a.s.},
\end{equation}
and since
\begin{equation*}
	\begin{array}{ll}
		&\ds{\sup_{t\in[0,T]}\Big\lvert \int_{0}^{t}\Big(\abs{\varrho_{k}(s)}_{H^{1}}^{2}\varrho_{k}(s)-\abs{\varrho(s)}_{H^{1}}^{2}\varrho(s)\Big)ds \Big\rvert_{H}  }\\
		\vs
		&\ds{\quad \quad \quad \quad \leq \int_{0}^{T}\big\lvert \abs{\varrho_{k}(s)}_{H^{1}}^{2}-\abs{\varrho(s)}_{H^{1}}^{2} \big\rvert ds +\int_{0}^{T}\abs{\varrho(s)}_{H^{1}}^{2}\abs{\varrho_{k}(s)-\varrho(s)}_{H}ds },
	\end{array}	
\end{equation*}
we get
\begin{equation}\label{key3}
	\lim_{k\to\infty}\sup_{t \in\,[0,T]}\left|\int_{0}^{t}\langle \abs{\varrho_{k}(s)}_{H^{1}}^{2}\varrho_{k}(s),\psi\rangle_Hds-\int_{0}^{t}\langle \abs{\varrho(s)}_{H^{1}}^{2}\varrho(s),\psi\rangle_H\,ds\right|=0,\ \ \ \ \ \ \ \ \hat{\P}\text{-a.s.}
\end{equation}
Moreover, for every $\eta \in\,C([0,T];H^1)\cap L^2(0,T;H^2)$ and $s \in\,[0,T]$ we have
\begin{equation*}
	\begin{array}{l}
	\ds{\langle \big(\partial^2_x\eta(s)\cdot \eta(s)\big)\eta(s)\varphi,\psi\rangle_H=-\langle\big(\partial_x\eta(s),\eta(s)\big)\eta(s)\varphi,\psi^\prime\rangle_H}\\
	\vs 
	\ds{\quad \quad \quad \quad -\langle \big(\partial_x\eta(s),\eta(s)\big)\big(\partial_x\eta(s)\varphi+\eta(s)\varphi^\prime\big)+\vert\partial_x\eta(s)\vert^2\eta(s)\varphi,\psi\rangle_H.}	
	\end{array}
\end{equation*}
This implies that
\begin{equation*}
	\begin{array}{l}
		\ds{\int_{0}^{t}\langle \big(\partial^2_x\varrho_{k}(s)\cdot \varrho_{k}(s)\big)\varrho_{k}(s)\varphi-\big(\partial^2_x\varrho(s)\cdot \varrho(s)\big)\varrho(s)\varphi,\psi\rangle_H\,ds}\\
		\vs
	\ds{=-\int_0^t\langle\big[\big(\partial_x\varrho_k(s),\varrho_k(s)\big)\varrho_k(s)-\big(\partial_x\varrho(s),\varrho(s)\big)\varrho(s)\big]\varphi,\psi^\prime\rangle_H\,ds}\\
	\vs 
	\ds{\quad \quad-\int_0^t\langle\big[\big(\partial_x\varrho_k(s),\varrho_k(s)\big)\big(\partial_x\varrho_k(s)\varphi+\varrho_k(s)\varphi^\prime\big)-\big(\partial_x\varrho(s),\varrho(s)\big)\big(\partial_x\varrho(s)\varphi+\varrho(s)\varphi^\prime\big)\big],\psi\rangle_H\,ds	}\\
	\vs 
	\ds{\quad \quad\quad \quad \quad -\int_0^t\langle\big[\vert\partial_x\varrho_k(s)\vert^2\varrho_k(s)-\vert\partial_x\varrho(s)\vert^2\varrho(s)\big]\varphi,\psi\rangle_H\,ds	=:\sum_{i=1}^3 J_{i,k}(t).}
	\end{array}
	\end{equation*}
We have
\begin{equation*}\begin{array}{l}
\ds{\vert J_{1,k}(t)\vert\leq c\,\vert\varphi\vert_\infty\vert \psi^\prime\vert_\infty\int_0^T	\vert\varrho_k(s)-\varrho(s)\vert_{H^1}\vert\varrho_k(s)\vert^2_{H^1}\,ds}\\
\vs 
\ds{\quad \quad \quad \quad +c\,\vert\varphi\vert_\infty\vert \psi^\prime\vert_\infty\int_0^T\vert\varrho_k(s)-\varrho(s)\vert_{H^1}\vert \varrho(s)\vert_{H^1}\left(\vert \varrho_k(s)\vert_{H^1}+\vert \varrho(s)\vert_{H^1}\right)\,ds,}	
\end{array}
\end{equation*}
and, thanks to \eqref{convergence_as}, \eqref{bound2} and \eqref{bound3}, we conclude that 
\begin{equation}
\label{sm-fine1}
\lim_{k\to\infty}	\sup_{t \in\,[0,T]}\,|J_{1,k}(t)|=0,\ \ \ \ \ \ \hat{\mathbb{P}}-\text{a.s.}.
\end{equation}	
Next, we have
\begin{equation*}\begin{array}{l}
\ds{\vert J_{2,k}(t)\vert\leq c\,\vert \psi\vert_\infty\int_0^T\vert\varrho_k(s)-\varrho(s)\vert_{H^1}\left(\vert\varrho_k(s)\vert_{H^1}+\vert\varrho(s)\vert_{H^1}\right)\left(\vert\varrho_k(s)\vert_{H^2}\vert\varphi\vert_\infty+\vert\varrho_k(s)\vert_{H^1}\vert\varphi^\prime\vert_\infty\right)\,ds}\\
\vs 
\ds{\quad \quad \quad \quad \quad \quad \quad \quad +c\,\vert \psi\vert_\infty\left(\vert\varphi\vert_\infty+\vert \varphi^\prime\vert_\infty\right)\int_0^T\vert\varrho_k(s)-\varrho(s)\vert_{H^1}\vert\varrho(s)\vert_{H^2}\vert\varrho(s)\vert_{H^1}\,ds.}
\end{array}
\end{equation*}		
Thus, as a consequence of \eqref{convergence_as} and bounds \eqref{bound1}, 	\eqref{bound2}, \eqref{bound3} and \eqref{bound7}, we can conclude that
\begin{equation}
\label{sm-fine2}
\lim_{k\to\infty}\hat{\mathbb{E}	}\sup_{t \in\,[0,T]}\vert J_{2,k}(t)\vert=0.
\end{equation}
Finally, we have	
\begin{equation*}
\begin{array}{l}
\ds{\vert J_{3,k}(t)\vert \leq c\,\vert\varphi\vert_\infty\vert \psi\vert_\infty\int_0^T\vert\varrho_k(s)-\varrho(s)\vert_{H^1}\left[\vert\varrho_k(s)\vert_{H^1}\left(\vert\varrho_k(s)\vert_{H^2}+\vert \varrho(s)\vert_{H^2}\right)+\vert \varrho(s)\vert_{H^1}^2\right]\,ds,}
\end{array}
	\end{equation*}
	and, due again to \eqref{convergence_as}, \eqref{bound1} and\eqref{bound7}, we conclude
	\begin{equation}
	\label{sm-fine3}	
	\lim_{k\to\infty}\hat{\mathbb{E}	}\sup_{t \in\,[0,T]}\vert J_{3,k}(t)\vert=0.
\end{equation}

Therefore, as a consequence of \eqref{sm-fine1}, \eqref{sm-fine2} and \eqref{sm-fine3}, we conclude that
\begin{equation}
\label{sm-fine4}
\lim_{k\to\infty}	\hat{\mathbb{E}}\sup_{t \in\,[0,T]}\left|\int_{0}^{t}\langle \big(\partial^2_x\varrho_{k}(s)\cdot \varrho_{k}(s)\big)\varrho_{k}(s)\varphi-\big(\partial^2_x\varrho(s)\cdot \varrho(s)\big)\varrho(s)\varphi,\psi\rangle_H\,ds\right|=0.
\end{equation}

Next, for every $t \in\,[0,T]$ we have
\begin{equation*}
	\begin{array}{l}
	\ds{ \Big\lvert \varphi \int_{0}^{t}\Big(\abs{\varrho_{k}(s)}_{H^{1}}^{2}\abs{\varrho_{k}(s)}^{2}\varrho_{k}(s)-\abs{\varrho(s)}_{H^{1}}^{2}\abs{\varrho(s)}^{2}\varrho(s) \Big)ds \Big\rvert_{H} }\\
	\vs
	\ds{\leq c\vert\varphi\vert_{\infty}\Bigg( \int_{0}^{T}\big\lvert\abs{\varrho_{k}(s)}_{H^{1}}^{2}-\abs{\varrho(s)}_{H^{1}}^{2}\big\rvert\cdot \big\lvert \abs{\varrho_{k}(s)}^{2}\varrho_{k}(s)\big\rvert_{H}ds }\\
	\vs
	\ds{\quad\quad+\int_{0}^{T}\abs{\varrho(s)}_{H^{1}}^{2}\big\lvert (\abs{\varrho_{k}(s)}^{2}-\abs{\varrho(s)}^{2})\varrho_{k}(s)\big\rvert_{H}ds  +\int_{0}^{T}\abs{\varrho(s)}_{H^{1}}^{2}\big\lvert \abs{\varrho(s)}^{2}(\varrho_{k}(s)-\varrho(s))\big\rvert_{H}ds   \Bigg) }\\
	\vs
	\ds{\leq c\vert\varphi\vert_{\infty}\Big( \int_{0}^{T}\abs{\varrho_{k}(s)}_{H^{1}}^{2}\big(\abs{\varrho_{k}(s)}_{H^{1}}+\abs{\varrho(s)}_{H^{1}}\big)\abs{\varrho_{k}(s)-\varrho(s)}_{H^{1}}ds }\\
	\vs 
	\ds{\quad\quad\quad \quad\quad\quad\quad\quad\quad+\int_{0}^{T}\big(\abs{\varrho_{k}(s)}_{H^{1}}^{4}+\abs{\varrho(s)}_{H^{1}}^{4}\big)\abs{\varrho_{k}(s)-\varrho(s)}_{H}ds  \Big).  }
	\end{array}
\end{equation*}
Thus, thanks again to \eqref{bound2} and \eqref{bound3}, we get
\begin{equation}\label{key5}
	\lim_{k\to\infty} \sup_{t \in\,[0,T]}\left|\int_{0}^{t}\langle \varphi\abs{\varrho_{k}(s)}_{H^{1}}^{2}\abs{\varrho_{k}(s)}^{2}\varrho_{k}(s),\psi\rangle_H\,ds- \int_{0}^{t}\langle \varphi\abs{\varrho(s)}_{H^{1}}^{2}\abs{\varrho(s)}^{2}\varrho(s),\psi\rangle_H\,ds\right|=0,\ \ \ \ \hat{\P}\text{-a.s.}
\end{equation}
By using the same arguments as in the proof of Lemma \ref{reminder}, we conclude that
\begin{equation}\label{key6}
	\lim_{k\to\infty}\hat{\mathbb{E}}\sup_{t \in\,[0,T]}\left|\langle\widehat{R}_{k}(t),\psi\rangle_H\right|^2 = 0.
\end{equation}
Finally, combining  \eqref{key1}, \eqref{key2}, \eqref{key3}, \eqref{sm-fine4} and \eqref{key5} together with \eqref{key6}, we can conclude that $\varrho$ satisfies equation \eqref{limiting_equation}. As we have seen above, this allows to conclude the proof of Theorem \ref{limite}.

\end{proof}

\end{document}